\documentclass[11pt,a4paper]{article}

\usepackage{authblk}

\usepackage[english]{babel}
\usepackage{graphicx} 
\usepackage{amsthm} 
\usepackage{amsmath}
\usepackage{amssymb}
\usepackage{mathtools}
\usepackage{amsthm} 
\usepackage{wasysym} 
\usepackage{mathrsfs} 
\usepackage{bbm} 
\usepackage{calligra} 
\usepackage{enumerate} 
\usepackage[shortlabels]{enumitem}
\usepackage{xcolor}
\usepackage{multirow} 
\usepackage{subfigure}
\usepackage{nicefrac}
\usepackage{comment}
\usepackage{xfrac}
\usepackage{esint}
\usepackage{hyperref}
\usepackage{bbm}

\newcommand{\w}{\omega}
\newcommand{\eps}{\varepsilon}
\newcommand{\pa}{\partial}
\newcommand{\na}{\nabla}
\newcommand{\N}{\mathbb{N}}
\newcommand{\Z}{\mathbb{Z}}
\newcommand{\Sym}{\text{Sym}_{2\times 2}}
\newcommand{\ProL}{\mathbb{P}}

\newcommand{\R}{\mathbb{R}}

\newcommand{\T}{\mathbb{T}}
\newcommand{\Hp}{\mathcal{H}^p}

\newcommand{\fancyD}{\mathcal{D}}
\newcommand{\fancyF}{\mathcal{F}}
\newcommand{\fancyR}{\mathcal{R}}
\newcommand{\fancyS}{\mathcal{S}}

\def\norm#1{\left\vert \left\vert #1  \right\vert \right\vert }

\newcommand{\bpf}{\begin{proof}}
\newcommand{\epf}{\end{proof}}

\DeclareMathOperator{\divergence}{div}

\DeclareMathOperator{\curl}{curl}
\DeclareMathOperator{\supp}{supp}

\newcommand{\divinv}{\divergence^{-1}}

\newtheorem{thm}{Theorem}[section]
\newtheorem{prop}[thm]{Proposition}
\newtheorem{cor}[thm]{Corollary}
\newtheorem{lemma}[thm]{Lemma}
\newtheorem{defi}[thm]{Definition}
\newtheorem{rem}[thm]{Remark}

\title{Pathological solutions of Navier-Stokes equations on $\T^2$ with gradients in Hardy spaces}

\author[$a,b$]{Jan Burczak}
\author[$b$]{Antonio Hidalgo-Torn\'e}

\affil[$a$]{Institut f\"ur Mathematik, Universit\"at Leipzig}
\affil[$b$]{MPI f\"ur Mathematik in den Naturwissenschaften, Leipzig}

\begin{document}
	
\maketitle

\begin{abstract}
For an arbitrary smooth initial datum, we construct multiple nonzero solutions to the $2$d Navier-Stokes equations, with their gradients in the Hardy space $\mathcal{H}^p$ with any $p \in (0,1)$. Thus, in terms of the path space $C(\mathcal{H}^p)$ for vorticity, $p=1$ is the threshold value distinguishing between non-uniqueness and uniqueness regimes. In order to obtain our result, we develop the needed theory of Hardy spaces on periodic domains. 
\end{abstract}

\section{Introduction}
Denote by $\T$ the periodic interval $[-\frac{1}{2},\frac{1}{2}]$. Consider the incompressible Navier-Stokes equations (NSE) on $\T^2  \times (0,T)$:
\begin{equation}\label{eq:NSE}
\begin{split}
\pa_t u+\divergence(u\otimes u)- \Delta u +\na \pi&=0, \\
\divergence u&=0.
\end{split}
\end{equation}
By $\Hp$ we denote the real Hardy space on $\T^2$, see  Appendix for the definition\footnote{Our definition of Hardy space on $\T^n$ is a natural modification of the definition for $\R^n$.}. The vorticity of $u$ is denoted by $\omega$, i.e.\ $\curl u =:\omega$.
\begin{defi}
\emph{A very weak solution to \eqref{eq:NSE}} is  $u \in L^2 (\T^2  \times (0,T)) \cap C([0,T],L^1(\T^2))$ with zero-mean on $\T^2$  solving 
distributionally \eqref{eq:NSE}.
\end{defi}

We prove the following dichotomy for very weak solutions of \eqref{eq:NSE}, in terms of the index $p$ of the path space $C([0,T],\mathcal{H}^{p})$ for vorticity.

\begin{thm}\label{thm:main}
\begin{itemize}
    \item[(A)] Well-posedness: Fix $p \ge 1$. In terms of $\omega$, $C([0,T],\mathcal{H}^{p})$ is a well-posedness path space for \eqref{eq:NSE}. More precisely, if for two very weak solutions $u_1, u_2$ to \eqref{eq:NSE}, such that their vorticities belong to $C([0,T],\mathcal{H}^{p})$, holds $u_1 =u_2$ at $t=0$, then $u_1 \equiv u_2$.

 \item[(B)] Ill-posedness: Fix $p\in (0,1)$. In terms of $\omega$, $C([0,T],\mathcal{H}^{p})$ is a ill-posedness space for \eqref{eq:NSE}. More precisely, for any smooth initial datum $a_{in}$ there are infinitely many very weak solutions $\{u_i\}_{i \in \N}$ to \eqref{eq:NSE}, such that their gradients belong to $C([0,T],\mathcal{H}^{p})$, ${u_i}_{|t=0} = a_{in}$ $\forall_{i}$, but $u_i \neq u_j$ for $i \neq j$. (In fact, these solutions belong also to $C([0,T],W^{\sigma,1})$ for an arbitrary $\sigma<1$)
\end{itemize}
\end{thm}

The proof of the well-posedness side is almost immediate: $\omega \in C([0,T],\mathcal{H}^{p})$ for $p\ge1$, with $u \in C([0,T],L^1(\T^2))$ yields $u\in C([0,T],L^2(\T^2))$ thanks to Proposition \ref{prop:H1inL2}. This means that Ladyzhenskaya-Prodi-Serrin condition applies (see for instance Theorem 1.3 of \cite{cheskidovluo2022}), yielding uniqueness\footnote{Since there are regular solutions emanating from regular data for 2D NSE, our well-posedness statement is not empty. It could be interesting to prove that initial vorticity in $\mathcal{H}^{p}$, in particular for $p=1$, gives rise to mild solution with vorticity in $C([0,T],\mathcal{H}^{p})$.}.

In fact, we prove a result stronger than Theorem \ref{thm:main} (B), namely the following result on existence of pathological solutions 

\begin{thm}\label{thm:main2}
Take any two smooth solutions $v_1, v_2$ to \eqref{eq:NSE}. Fix $p\in (0,1)$ and $\sigma\in (0,1)$.
There exists non-zero very weak solution $u$ to \eqref{eq:NSE}, $u \in C([0,T],W^{\sigma,1}) \cap L^2 (\T^2  \times (0,T))$, with $\nabla u\in C([0,T],\mathcal{H}^{p})$, such that: $u= v_1$ on time interval $[0, \sfrac{1}{8}]$ and  $u= v_2$ on time interval $[\sfrac{7}{8}, 1]$.
\end{thm}
Theorem \ref{thm:main} follows from Theorem \ref{thm:main2} by taking $v_1$ to be smooth solution to \eqref{eq:NSE} with its initial datum $a_{in}$, and taking $v_2$ to be any (other) smooth solution to \eqref{eq:NSE} (for instance, starting from initial datum $b_{in}$, such that $\|b_{in}\|_{L^2} = \frac{1}{2} \|{v_1}_{|t=1}\|_{L^2}$).

\begin{rem}
(i) Our results show there is a 'pathological solution' (in our low-regularity path space) to \eqref{eq:NSE} for any smooth datum. In particular, the criterion distinguishing between 'well-posedness' and 'ill-posedness' regimes cannot be stated solely in terms of (regularity of) space of initial data, but must involve a desired path space\footnote{This comment is aimed at clarifying the way Theorem \ref{thm:main} is stated, and explaining why 'optimal regularity result' of type mentioned in the preceding footnote is not our focus. It seems that some authors present similar non-uniqueness results merely in terms of the space of initial data.}. \\
(ii) The 'genericity' of our pathological solutions can be expressed in a sharper way than 'infinitely many emanating from arbitrary $a_{in}$', using Baire category approach.
\end{rem}

Besides providing a sharp threshold for well/ill-posedness of $2$D NSE, our paper develops a toolbox for convex integration in Hardy spaces on $\T^n$ (see Appendix), which we hope to be useful in other contexts. 
The advantage of Hardy spaces over Lebesgue setting is striking in the range $0<p< 1$, where $L^p$ sets are useless, whereas $\mathcal{H}^{p}$ spaces retain most of the desired properties, and allow to mimick '$p$-scaling' considerations. Even in the case $p=1$, due to Calder\'on-Zygmund theory, $\mathcal{H}^{1}$ may be preferable over $L^1$.

\subsection{Discussion and background}

Our work is inspired by the important result of Cheskidov\&Luo \cite{cheskidovluo2022}, \cite{cheskidovluo2023}, which indicate that convex integration for NSE can critically benefit from time intermittency.
Furthermore, the best to date regularity known for non-unique solutions of 2D Navier-Stokes equations has been obtained in those papers: (a) $u\in C((0,T];W^{{\sigma}',1})$ for any ${\sigma}'<1$ in \cite{cheskidovluo2023}, or (b) $u\in L^1(0,T;W^{1,q})$, $q<\infty$ in \cite{cheskidovluo2022}. It seems that our result is the first one that manages to control the full derivative uniformly in time: namely, on top of the just-mentioned regularity class (a) appearing in \cite{cheskidovluo2023}, our non-unique solution enjoys $\nabla u \in C([0,T]; \Hp)$ with $p$ arbitrarily close to $1$, see Theorem \ref{thm:main2}.

A loose idea to use Hardy spaces in a convex integration scheme surfaced in the discussions of the first author with Stefano Modena over \cite{BrueColombo23_L1infty}. This resulted in \cite{buckmodena2024}, \cite{buckmodena2024compactly} of  Buck\&Modena, considering 2D Euler equations. More precisely, in \cite{buckmodena2024} the authors prove non-uniqueness of weak solutions of the 2d Euler equations with $\curl u \in C(\Hp(\R^2))$ with $2/3<p<1$, and improve in \cite{buckmodena2024compactly} to any $0<p<1$, obtaining there also much better control of supports. Let us mention in passing that the result  \cite{BrueColombo23_L1infty} has been improved, see \cite{BrueColomboKumar}.

Even though the periodic setting is the simplest one, the space $\Hp(\T^n)$ seems not even properly defined in the literature, so we decided to systematically develop the needed theory in Appendix. It turns out that the space $\Hp(\T^n)$ has several useful advantages of the periodic (compact) setting over the full-space case, including:
\begin{enumerate}
    \item The embedding $\Hp(\T^n)\subset \mathcal{H}^q(\T^n)$ for any $0<q<p\leq \infty$,
    \item The fact that not all atoms must have vanishing moments. Consequently, the constraint $2/3<p<1$ can be removed even within the methodology of \cite{buckmodena2024}, and certain additional technical difficulties, specific for the NSE case, can be circumvented.
\end{enumerate}

Let us mention that the convex integration methods for incompressible fluid dynamics are well established, see in particular \cite{Laszlo, Isett, BuckmasterVicol, BCV}, and the recent reviews \cite{rev1, rev2}. 
There are another exciting approaches to non-uniqueness, including \cite{Vishik1, Vishik2}, see also \cite{bookABC+, spanish}; \cite{JiaSverak, ABC}; \cite{ems}, and \cite{CP}.

\subsection{Notation}
In this subsection we fix some notation.

\begin{itemize}
\item We write $a\lesssim b$ if
there exists a constant $C>0$ such that $a\leq Cb$, and the constant does not depend on important parameters (for instance, in a convex integration step, it depends only on quantities fixed by the inductive assumption).

\item We denote by $C^\infty_0(\T^2,\R^2)$ the space of zero mean smooth periodic functions with values in $\R^2$.

\item We denote by $\Sym$ the set of real $2\times 2$ symmetric matrices.

\item We denote by $\ProL$ the Leray projector acting on $\T^n$.

\end{itemize}

\section*{Acknowledgement} We would like to thank Adam Black and Pawe{\l} Goldstein for useful exchange concerning Hardy spaces.

\section{Non-uniqueness}\label{sec:mip}
The non-uniqeness will be achieved by invoking iteratively the key proposition, presented in this subsection. It allows to reduce the error in an approximate solution to \eqref{eq:NSE}, called here the NS-Reynolds system (or, briefly, NSR).

\begin{defi}[Solution to NS-Reynolds system]
A solution to the NS-Reynolds system (NSR) is a tuple $(u,p,R)$ of smooth functions 
$$u\in C^\infty([0,1]\times\T^2;\R^2),
\quad \pi\in C^\infty([0,1]\times\T^2), \quad R\in C^\infty([0,1]\times\T^2;\textnormal{Sym}_{2\times 2})$$
such that 
\begin{equation}\label{eq:NSR}
\begin{split}
\pa_t u+\divergence(u\otimes u) - \Delta u+\na \pi&=-\divergence {R}, \\
\divergence u&=0. 
\end{split}
\end{equation}
    
\end{defi}

\begin{prop}[Key proposition]\label{prop:main}
Choose any $\delta, \tau>0$, {$0\leq \sigma<1$, }and $0<p<1$. Assume that there exists a smooth solution $(u_0,R_0,\pi_0)$ to \eqref{eq:NSR}.
Then there exists another smooth solution $(u_1,R_1,\pi_1)$ to \eqref{eq:NSR} such that 
\begin{subequations}\label{eq:mainsubequations}
\begin{equation}\label{eq:mainpropR}
\norm{R_1}_{L^1_{t,x}}<\delta,
\end{equation}
\begin{equation}\label{eq:mainpropu}
\norm{u_1-u_0}_{L^2_{t,x}}\le M \norm{\mathring{R}_0}^\frac{1}{2}_{L^1_{t,x}} + \delta,
\end{equation}
\begin{equation}\label{eq:mainpropcurl}
\sup_t \norm{\na(u_1-u_0)}_{\Hp}< \delta,
\end{equation}
\begin{equation}\label{eq:mainpropCL}
{\sup_t \norm{u_1-u_0}_{W^{\sigma,1}}< \delta,}
\end{equation}
\end{subequations}
where $M = 8\sqrt{10}$.

Furthermore, if $R_0\equiv 0$ on $[0,t_0] \cup [1-t_0,1]$ for some $t_0 \in (0, \frac{1}{2})$, then  $R_1\equiv 0$ on $[0,t_0-\tau] \cup [1-t_0+\tau,1]$ and $u_0 \equiv u_1$ on $[0,t_0-\tau] \cup [1-t_0+\tau,1].$
\end{prop}
The proof of Proposition \ref{prop:main} is deferred to the following sections. Now we prove Theorem \ref{thm:main2} having Proposition \ref{prop:main}.

\begin{proof}[Proof of Theorem \ref{thm:main2}] Recall that $v_1, v_2$ are given smooth solutions to \eqref{eq:NSE}.
Take a nonzero, smooth, divergence-free $f:\T^2\times [0,1] \to \T^2$ such that $\int_{\T^2}f(t,x)dx=0 $ for any $t$, and $f(t) \equiv 0$ on $[0,1/3] \cup [2/3,1]$. Take two (time) cutoff functions $\eta_1, \eta_2:  [0,1] \to [0,1]$ such that $\eta_1 =1$ on $[0,1/4]$ and vanishes outside $[0,1/3]$, and $\eta_2 =1$ on $[3/4,1]$ and vanishes outside $[2/3,1]$. Thus 
\begin{equation}\label{eq:u0}
u_0 :=v_1 \eta_1 + f + v_2 \eta_2 =\begin{cases}
 v_1 & \text{for }  t \in [0,1/4], \\
f & \text{for }  t \in [1/3, 2/3], \\
v_2 & \text{for }  t \in  [3/4,1].
\end{cases}    
\end{equation}

We will apply Proposition \ref{prop:main} iteratively, with $\tau_n=2^{-n-4}$ and $\delta_n= C 2^{-n}$, where $C= \frac{1}{10 M}\min (1, \sup_t \norm{\nabla f}_{\Hp})$.
The starting tuple is $(u_0,R_0,\pi_0)$, with $u_0$ given by \eqref{eq:u0}, 
\[R_0 =-\divergence^{-1} (\partial_t u_0)-u_0\otimes u_0+\na u_0 {+\na u_0^T},\] where $\divergence^{-1}$ is symmetric antidivergence operator, see Lemma \ref{lem:antidivergenceoperators}, and $\pi_0 = v_1 \eta_1 + v_2 \eta_2$.
Observe that $R_0= 0$ on $[0,1/4]\cup [3/4,1]$.

Iterating Proposition \ref{prop:main}, we obtain a sequence
$(u_n,R_n,p_n)$, with $u_n$ being a Cauchy sequence in $L^2_{t,x} \cap C_t W_x^{s,1}$ and $\nabla u_n$ Cauchy in $C_t \Hp$, whereas equation \eqref{eq:mainpropR} implies that $R_n\to 0$ in $L^1_{t,x}$. Therefore the limit $u$ is a weak solution to Navier Stokes equations \eqref{eq:NSE} with the desired regularity. Observe that $u \neq 0$ no matter what are $v_1, v_2$ (in particular, for $v_1= v_2=0$), because by choice of $C$ we ensured that $\sup_t \norm{\nabla u}_{\Hp} \ge \sup_{t \in [1/3, 2/3]}\norm{\nabla u}_{\Hp} \ge \frac{9}{10}\sup_t \norm{\nabla f}_{\Hp}$. The last property of Proposition \ref{prop:main} and our choice of  $\tau_n$ ensures that $u \equiv v_1$ on $[0,1/8]$ and that $u \equiv v_2$ on $[7/8,1]$
\end{proof}

The rest of the article (except for Appendix) is aimed at proving Proposition \ref{prop:main}

\section{Preliminaries}
In order to decrease the error $R_0$
we will modulate a function on $\T^2$ (say $f$) with a concentrated, fast oscillating function. The latter arises as follows: given a function $g:\R\to\R$ supported on $(-\frac{1}{2},\frac{1}{2})$, we concentrate it, i.e.\ write $\mu^\frac{1}{2}g(\mu \cdot)$, with a parameter $\mu \ge 1$ and then periodise it with period $1$. The resulting function is denoted by adding subscript $\mu$, i.e.\ $g_\mu$. Observe that $g_\mu$ is intermittent, i.e. $\norm{\nabla^l g_\mu}_{L^s(\T)}=\mu^{l+\frac{1}{2}-\frac{1}{s}}\norm{g}_{L^s(\T)}$, which will allow us to control high derivatives at the cost of low integrability.
As usual in convex integration schemes, we will invoke fast oscillations to decrease error. Therefore, for $\phi \in C^\infty(\T^2)$ and $\lambda\in \N$, we will consider the fast-oscillating function $\phi(\lambda \cdot)$. Observe $\norm{\nabla^l\phi (\lambda \cdot)}_{L^s(\T^2)}=\lambda^l\norm{\nabla^l\phi}_{L^s(\T^2)}$. An analogous procedure will be performed for time concentrations and oscillations.
The concrete realisations of the concentrated, fast oscillation function will be presented when we introduce our building blocks. 

Now we introduce needed auxiliary results.
\subsection{Technical tools}

The following antidivergence operators will be used 
\begin{lemma}[Antidivergence operators]\label{lem:antidivergenceoperators}
Let $\lambda\in \N$ and $p \in [1, \infty]$.
\begin{enumerate}[label=(\roman*)]
   \item ($\divergence^{-1}$: symmetric antidivergence) There exists $\divinv:C^\infty_0(\T^2;\R^2)\to C^\infty_0(\T^2;\Sym)$ such that $\divergence\divinv u=u$ and for $i\geq 0$ one has
   $$\norm{\na^i \divinv u}_{L^p}\lesssim \norm{\na^i u}_{L^p},$$
    \item ($\fancyR$: symmetric bilinear antidivergence) There exists a bilinear operator $\fancyR:C^\infty(\T^2;\R)\times C^\infty_0(\T^2;\R^2)\to C^\infty_0(\T^2;\Sym)$ such that $\divergence\fancyR(f,u)=fu-\fint fu$ and 
    $$\norm{\fancyR(f,u_\lambda)}_{L^p}\lesssim \frac{1}{\lambda}\norm{u}_{L^p}\norm{f}_{C^1}.$$
    \item ($\tilde{\fancyR}$: symmetric bilinear antidivergence on tensors) There exists a bilinear operator $\tilde{\fancyR}:C^\infty(\T^2;\R^2)\times C^\infty_0(\T^2;\R^{2\times 2})\to C^\infty_0(\T^2;\Sym)$ such that $\divergence\tilde{\fancyR} (v,T_\lambda)=Tv-\fint Tv$ and 
    $$\norm{\tilde{\fancyR}(v,T_\lambda)}_{L^p}\lesssim \frac{1}{\lambda} \norm{T}_{L^p}\norm{v}_{C^1}.$$
\end{enumerate}
\end{lemma}
The above result is a simplified version of \cite[Proposition 4]{BMS2021}.

Let us fix directions
\begin{equation}\label{eq:dirs}
\xi_1=(1,0), \; \xi_2=(0,1),  \;\xi_3=(1,1),  \;\xi_4=(1,-1).    
\end{equation}
The following Lemma will be used to decompose the error, see \cite[Section 5]{BrueColombo23_L1infty}.

\begin{lemma}[Nash decomposition]\label{lem:Nashdecomposition}
There exist $\Gamma_k\in C^\infty (\Sym, \R), k=1,\dots,4$ such that for any $R\in \Sym$ such that $|R-I|<\frac{1}{8}$
$$R=\sum_{i=k}^4 \Gamma^2_k(R)\frac{\xi_k}{|\xi_k|}\otimes \frac{\xi_k}{|\xi_k|}.$$
\end{lemma}

 One of our purposes is to develop convex-integration-related toolbox in Hardy spaces on $\T^n$. This is done in the appendix. Now we distill those results from the appendix, which are directly needed in this paper.

\begin{prop}\label{cor:mainhardyembedding}
For any $0<p\leq q\leq \infty,$
\begin{equation*}
\norm{f}_{\Hp}\leq \norm{f}_{\mathcal{H}^q}.
\end{equation*} 
\end{prop}

\begin{prop}\label{prop:mainHpcomplete}
For $0<p\leq 1$, $\Hp$ is a complete metric space with the metric given by $d(f,g)=\norm{f-g}_{\Hp}^p$.
\end{prop}

\begin{defi}[$\Hp$-atoms]\label{def:mainatom}
Let $a\in \fancyD'(\T^n)$ be a measurable function, and $\bar{B}_a$ its representative ball according to Definition \ref{def:Rnball}. We say that $a$ is an $\Hp$-atom associated to $\bar{B}_a$ if:
\begin{enumerate}
    \item $\norm{{a}}_{L^\infty}\leq |\bar{B}_{a}|^{-\frac{1}{p}}$.
    \item If $\bar{B}_{a}$ is a small ball, $\int_{\overline{B}_{a}} x^\beta {a}=0$ for all multiindices $ \beta$ with $|\beta|\leq n(p^{-1}-1)$. 
\end{enumerate} 
\end{defi}
See definitions \ref{def:Rnball}, \ref{def:smallball} for rigorous meaning of the 'representative ball' and 'small ball'.

\begin{cor}\label{cor:mainLinftyinHp}
A function $f\in L^\infty(\T^n)$ satisfying condition \ref{itematomscancellations} in Definition \ref{def:mainatom}, satisfies the bound 
$$\norm{f}_{\Hp}\leq C|\bar{B}_f|^{\frac{1}{p}}\norm{f}_{L^\infty}.$$
\end{cor}

\begin{prop}\label{prop:mainHpboundnonvanishingatoms}
Let $0<p\leq 1$. Let $f\in L^\infty(\T^n)$ be a function, and $\bar{B}_f$ its representative ball according to Definition \ref{def:Rnball}. Assume that $\bar{B}_f$ has center $\bar{x}$ and radius $\eps$, with $0<\eps<\frac{1}{2}$. Let $N= \lfloor n(\frac{1}{p}-1)\rfloor$. Then 
\begin{align*}
\norm{f}_{\Hp}^p\leq &C(p,n)\big(\eps^n\norm{f}_{L^\infty}^p
+|\log\eps|\max_{|\alpha|\leq N}\left|\int_{\bar{B}_f} x^\alpha f^\text{ext}(x) dx \right|^p\big),
\end{align*}
where $f^\text{ext}$ denotes the periodic extension of $f$.
\end{prop}

\begin{prop}\label{prop:mainPboundedHp}
Let $\ProL$ be the Leray projector, $0<p<\infty$, and $f\in \Hp$. Then, $$\norm{\ProL f}_{\Hp}\lesssim  \norm{ f}_{\Hp}.$$
\end{prop}

\begin{prop}\label{prop:H1inL2}
Let $f$ be such that {$\fint f=0$}, $\curl f\in \mathcal{H}^1$, and $\divergence f=0$. Then $f\in L^2$ and 
$$\norm{f}_{L^2}\lesssim \norm{\curl f}_{\mathcal{H}^1}.$$
\end{prop}
\section{Correctors}
In this section we define the correctors that we use to decrease the error in each iteration. First of all, we fix the Hardy space parameter $0<p<1$ and choose $N=\lfloor 2(\frac{1}{p}-1)\rfloor$. 
\subsection{Spatial building blocks}
Let us briefly recall {the} construction of \cite{buckmodena2024compactly}.
Take a smooth function $\Phi$ supported on $(-\frac{1}{2},\frac{1}{2})$ with zero mean. Define $\varphi=\nabla^{(2N+3)}\Phi$. Observe that $\varphi$ has zero mean, it is supported on $(-\frac{1}{2},\frac{1}{2})$, and by rescaling $\Phi$ we ensure  $\int_\R \varphi^2 = 1$. We define 
$$\varphi^k_\mu(x)=\varphi_{\mu}(x-\frac{k}{16}|\xi_k|^2),$$
with $\xi_k$ as in \eqref{eq:dirs}. The shifts in the above definition will facilitate disjointness of supports of the upcoming building blocks. 

We introduce now the functions
\begin{equation*}
\begin{aligned}
\alpha_k(x)=&\frac{1}{(\lambda\mu_2)^{2N+3}}\varphi_{\mu_1}^k(\lambda x_1)\Phi_{\mu_2}(\lambda x_2),\\
v_k(x)=&\pa_2^{2N+3}\alpha_k(x)=\varphi_{\mu_1}^k(\lambda x_1)\varphi_{\mu_2}(\lambda x_2),\\
q_k(x)=&\frac{1}{\w}(\varphi^k_{\mu_1})^2(\lambda x_1)\varphi_{\mu_2}^2(\lambda x_2).
\end{aligned}
\end{equation*}
and the rotations
$$\Lambda_k:\R^2\to\R^2, x\to (\xi_k\cdot x,\xi_k^\perp\cdot x).$$
These preceding ingredients allow us to define the building blocks $W_k,Y_k\in C^\infty([0,T]\times \T^2; \R^2),$ $A_k\in C^\infty([0,T]\times \R^2;\R^{2\times 2})$ as follows
\begin{equation}\label{eq:BBdef}
\begin{aligned}
W_k(t,x)=&v_k(\Lambda_k(x-\w t\frac{\xi_k}{|\xi_k|^2}))\frac{\xi_k}{|\xi_k|},\\
Y_k(t,x)=&q_k(\Lambda_k(x-\w t\frac{\xi_k}{|\xi_k|^2}))\xi_k,\\
A_k(t,x)=&\frac{1}{|\xi_k|^{2N+3}}\alpha_k(\Lambda_k(x-\w t\frac{\xi_k}{|\xi_k|^2}))\frac{\xi_k}{|\xi_k|}\otimes \frac{\xi_k^\perp}{|\xi_k|}.
\end{aligned}
\end{equation}
Observe that since $\xi_k\in\N^2$, the above functions remain periodic.

\subsubsection{Properties of spatial building blocks}\label{sssec:sp}
The building blocks are $\lambda$ periodic, $\supp W_k=\supp Y_k=\supp A_k$ for any $k$, and $\supp Y_{k_1}\cap \supp Y_{k_2}=\emptyset$ for $k_1\neq k_2$ and $1\ll \lambda, \mu_1,\mu_2$. Furthermore,

\begin{enumerate}
    \item $\divergence(W_k\otimes W_k)=\pa_t Y_k$,
    \item $\int_{\T^2}W_k\otimes W_k(t,x)=\frac{\xi_k}{|\xi_k|}\otimes \frac{\xi_k}{|\xi_k|},$
    \item $\int_{\T^2}W_k(t,x)dx=0$,
    \item $\int_{\T^2}q_k=\frac{1}{\omega}$, $\int_{\T^2}Y_k=\frac{1}{\omega}\xi_k$,
\end{enumerate}
and the following estimates hold uniformly in $t$ for all $k\in \{1,2,3,4\}, l\in \N$, provided $\mu_2 \ge \mu_1$
    \begin{equation}\label{eq:BBestimates}
    \begin{aligned}
    \norm{\na^l W_k}_{L^s(\T^2)}\lesssim&\lambda^{l}\mu_1^{\frac{1}{2}-\frac{1}{s}}\mu_2^{l+\frac{1}{2}-\frac{1}{s}},\\
    \norm{\na^l Y_k}_{L^s(\T^2)}\lesssim&\w^{-1}\lambda^{l}\mu_1^{1-\frac{1}{s}}\mu_2^{l+1-\frac{1}{s}},\\
    \norm{\na^l A_k}_{L^s(\T^2)},\norm{\na^l A^T_k}_{L^s(\T^2)}\lesssim&\lambda^{l-(2N+3)}\mu_1^{\frac{1}{2}-\frac{1}{s}}\mu_2^{l-(2N+3)+\frac{1}{2}-\frac{1}{s}},\\
    \norm{\na^l \pa_t A_k}_{L^s(\T^2)},\norm{\na^l \pa_t A^T_k}_{L^s(\T^2)}\lesssim&\w\lambda^{l-(2N+2)}\mu_1^{\frac{3}{2}-\frac{1}{s}}\mu_2^{l-(2N+3)+\frac{1}{2}-\frac{1}{s}},\\
    \norm{\na^l \divergence A^T_k}_{L^s(\T^2)}\lesssim&\lambda^{l-(2N+2)}\mu_1^{\frac{3}{2}-\frac{1}{s}}\mu_2^{l-(2N+3)+\frac{1}{2}-\frac{1}{s}},\\
    \norm{\na^l \divergence \divergence  A_k}_{L^s(\T^2)}\lesssim&\lambda^{l-(2N+1)}\mu_1^{\frac{3}{2}-\frac{1}{s}}\mu_2^{l-(2N+2)+\frac{1}{2}-\frac{1}{s}}.
    \end{aligned}
    \end{equation}
The last two inequalities require an observation, that there appears at least one derivative with respect to first variable. The first four are direct consequence of the definitions of the building blocks. Detailed proofs of the properties of the building blocks can be found in \cite[Proposition 5.3, Proposition 5.4, Lemma 5.5]{buckmodena2024compactly}.
\subsection{Space-time perturbations}
Assume we have a solution $(u_0,\pi_0,R_0)$ to NSR \eqref{eq:NSR}. The trace of $R_0$ can be removed by redefining the pressure, thus we have another solution to \eqref{eq:NSR} $(u_0,\tilde \pi_0,\mathring{R}_0)$, denoting by $\mathring{R}_0$ the traceless part of $R_0$.
Recall that, thanks to Lemma \ref{lem:Nashdecomposition}, there are smooth functions $\Gamma_k$ with $|\Gamma_k|\leq 1$ such that for any matrix $A$ with $|A-I|<\frac{1}{8}$
$$A=\sum_{k=1}^4 \Gamma_k^2(A)\frac{\xi_k}{|\xi_k|}\otimes \frac{\xi_k}{|\xi_k|}.$$
In order to prove the last property of the Proposition \ref{prop:main} (the time support),  we take a smooth function $\Theta$.
In case we assume $R_0\equiv 0$ on $[0,t_0] \cup [1-t_0,1]$ for some $t_0>0$, we take
\begin{equation*}
\left\{
\begin{aligned}
\Theta(t)&\equiv 0 \quad &&\text{for\, }t \le t_0-\tau \text{ \,or\, } t \ge 1-t_0+\tau,\\
\Theta(t)&\equiv 1 \quad &&\text{otherwise,\, }
\end{aligned}
\right.
\end{equation*}
where $\tau>0$ is fixed in the statement of Proposition \ref{prop:main}. In case we don't assume $R_0\equiv 0$ on $[0,t_0]$, we take  $\Theta \equiv 1$.

For $\eps>0$ we define
\begin{align*}
\rho(t,x)=&10\sqrt{\eps^2+|\mathring{R}_0|^2},\\
a_k(t,x)=&\Theta(t)\rho^{\frac{1}{2}}\Gamma_k\left(I+\frac{\mathring{R}_0}{\rho}\right),
\end{align*} which yields the following decomposition of $\mathring{R}_0$
\begin{equation}\label{eq:adef}
\sum_k a^2 \frac{\xi_k}{|\xi_k|}\otimes \frac{\xi_k}{|\xi_k|}=\Theta^2(\rho I+ \mathring{R}_0).
\end{equation}
Now we implement the temporal concentration, which allows to deal satisfactorily with the dissipative part.
Let nonnegative $g\in C^\infty_c(0,1)$ be such that
\begin{equation}\label{eq:g_u2}
 \int_0^1g^2(t)dt=1.
\end{equation}
We define $g_\kappa(t)$ as the $1$-periodic extension of $\kappa^{\frac{1}{2}}g(\kappa t)$, so that 
\begin{equation}\label{eq:gestimates}
\norm{g_\kappa}_{L^p([0,1])}\le C \kappa^{\frac{1}{2}-\frac{1}{p}}.
\end{equation}
Let $h_\kappa$ be defined by 
$$h_\kappa(t)=\int_0^t (g^2_\kappa(s)-1) ds.$$
Note that $h_\kappa$ is  $1$-periodic (thanks to \eqref{eq:g_u2}), and bounded uniformly in $\kappa$.
For any $\nu\in \N$, the fast oscillating function $g_\kappa(\nu\cdot)$ has the same $L^p$-norm as $g_\kappa(\cdot)$, and
\begin{equation}\label{eq:h}
\pa_t(\nu^{-1}h_\kappa(\nu t))=g^2_\kappa(\nu t)-1.
\end{equation}

Our perturbation, aimed at decreasing $\mathring{R}_0$, will have three main components: (i) The principal perturbation
\begin{equation}\label{eq:up}
u^p=g_\kappa(\nu t)\sum_k a_kW_k,
\end{equation}
is chosen to reduce the $g^2_\kappa(\nu t) \mathring{R}_0 (t)$ part of the error, via $u^p \otimes u^p$. (ii) The time corrector 
\begin{equation}\label{eq:ut}
u^t=-g_\kappa^2(\nu t)\sum_ka_k^2(Y_k-\frac{1}{\omega}\xi_k),
\end{equation}
is chosen to cancel an unwanted quadratic error term generated by $u^p$. (iii) Finally, we need to cancel the remainder of $\mathring{R}_0$, i.e.\ $(1-g^2_\kappa(\nu t) )\mathring{R}_0(t)$. To this end, we choose the main component of the time-intermittency corrector to be 
\begin{equation}\label{eq:ug}
u^g (t)=   -\frac{h_\kappa(\nu t)}{\nu} \divergence \mathring{R}_0 . 
\end{equation}

These three main components of the perturbation need to be augmented by secondary components, rendering the entire perturbation divergence-free and with vanishing moments for concentrated parts (the latter needed to ensure sufficient bounds for $\Hp$). Therefore the full perturbation of $u_0$
is $w$, i.e.
$$u_1=u_0+w=u_0+w^p+w^t+w^{g},$$
with
\begin{align}
w^p&=g_{\kappa}(\nu t)\sum_k \nabla^\perp \curl \Delta^N [\divergence(a_k(A_k+A_k^T))],\label{eq:wp}\\
w^t&=\ProL(u^t)-\fint u^t dx,\\
w^g&=\ProL(u^g), \label{eq:wg}
\end{align}
where $\ProL$ is the Leray projector.
The unintuitive choice \eqref{eq:wp} ensures $\na w^p$ in $\Hp$, thanks to vanishing of high moments. It contains the intuitive choice  
$u^p$ (when all derivatives fall on $\Phi_{\mu_2} $ present in $A_k+A_k^T$ of \eqref{eq:wp}; the remaining terms are lower order, as can be seen in the formula for $u^c$), which is itself insufficient to deal with $\Hp$. We note that our choice implies $u_1-u_0\equiv 0$ on $[0,t_0-\tau] \cup [1-t_0+\tau,1].$

The splitting of the perturbation $w$ into the main components and the secondary ones is given in the following result
\begin{prop}\label{prop:defmaincomponentsandcorrectors}
It holds 
$$w^p=u^p+u^c, \quad w^t=u^t+u^{cc}, \quad w^{g}=u^{g}+u^{ccc},$$
with main components $u^p$, $u^t$, $u^{g}$, given by, respectively, 
\eqref{eq:up}, \eqref{eq:ut}, \eqref{eq:ug},
and secondary correctors equal to
\begin{align*}
u^c=&
g_\kappa(\nu t)\sum_k[\Delta^{N+1},a_kf_{k_1}]f_{k_2}\xi_k
-g_\kappa(\nu t)\sum_k \na \divergence(\Delta^N(a_k\divergence A_k))\\
&+g_\kappa(\nu t)\sum_k\na^\perp \curl \Delta^N(A_k\cdot \na a_k +\divergence(a_kA_k^T)),\\
u^{cc}=& - \na \Delta^{-1}\divergence u^t-\fint u^t dx,\\
u^{ccc}=
&-\na\Delta^{-1}\divergence u^g,
\end{align*}
where 
\begin{align*}
f_{k_1}=&\varphi_{\mu_1}^k(\lambda(\xi_k\cdot x-\omega t)),\\
f_{k_2}=& \frac{1}{|\xi_k|^{2N+3}(\lambda \mu_2)^{2N+2}}(\Phi')_{\mu_2}(\lambda\xi_k^\perp\cdot x).
\end{align*}
\end{prop}
\bpf
The splitting of $w^p$ is proved in \cite[Proposition 6.1]{buckmodena2024compactly}. The only difference here is that we multiply the terms appearing in \cite{buckmodena2024compactly} with the function $g_\kappa$. The rest follows directly from  the identity $\mathbb{P}(u)=u-\na\Delta^{-1}(\divergence u)$.
\epf

\subsection{Estimates for correctors}\label{sec:estimatesperturbations}
The proper choice of the parameters involved in the perturbation $\tau, \eps, \mu_1,\mu_2,\kappa, \omega, \nu, \lambda$ will imply Proposition \ref{prop:main}. The exact dependence with respect to the solution $(u_0,\mathring{R}_0,\tilde{\pi}_0)$ will be not relevant. To shorten slightly the notation, let us choose always $\varepsilon \le \tau$, so that in constant only $\varepsilon$ features.

From \eqref{eq:adef}, which defines $a$, we have
\begin{align}\label{eq:abound}
\norm{a_k}_{C^l_{t,x}}\leq& C(\mathring{R}_0,\eps,l).
\end{align}
\subsubsection{Estimates in Lebesgue spaces}\label{subsec:estimateslebesgue}
To estimate $u^p$, we combine \eqref{eq:gestimates} with \eqref{eq:BBestimates} 
to get, for any  $r,s\in [1,\infty]$,

\begin{align*}
\norm{u^p}_{L^r_tL_x^s}\leq& C(\mathring{R}_0,\eps)\kappa^{\frac{1}{2}-\frac{1}{r}}\mu_1^{\frac{1}{2}-\frac{1}{s}}\mu_2^{\frac{1}{2}-\frac{1}{s}},\\
\norm{\na u^p}_{L^r_tL_x^1}\leq& C(\mathring{R}_0,\eps)\kappa^{\frac{1}{2}-\frac{1}{r}}\lambda\mu_1^{-\frac{1}{2}}\mu_2^{\frac{1}{2}},\\
\norm{u^p}_{L_{t,x}^2}\leq& 4\sqrt{10}(\|\mathring{R}_0\|_{L^1_{t,x}}^\frac{1}{2}+\eps^{\frac{1}{2}})+C(\mathring{R}_0,\eps )\lambda^{-\frac{1}{2}},
\end{align*}
where the last one requires a computation involving decorrelation, see \cite[Lemma 7.2]{buckmodena2024compactly}.
Similarly, from \eqref{eq:gestimates} and following the proof of \cite[Lemma 7.3]{buckmodena2024compactly}, we obtain for any $r,s\in [1,\infty]$ 

\begin{align*}
\norm{u^c}_{L^r_tL^s_{x}}\leq& C(\mathring{R}_0,\eps )\kappa^{\frac{1}{2}-\frac{1}{r}}\mu_1^{\frac{3}{2}-\frac{1}{s}}\mu_2^{-\frac{1}{2}-\frac{1}{s}},\\
\norm{\na u^c}_{L^r_tL^1_x}\leq& C(\mathring{R}_0 ,\eps )\kappa^{\frac{1}{2}-\frac{1}{r}}\lambda\mu_1^{-\frac{1}{2}}\mu_2^{\frac{1}{2}}.
\end{align*}
For the gradient estimate, we have also used that $u^c=w^p-u^p$, the triangle inequality, and \eqref{eq:BBestimates}.

We now bound the temporal perturbation $w^t$. The Leray projector $\ProL$ is bounded from $L^s$ to $L^s$  for any $1<s<\infty$, so for bounding $w^t$ it suffices to bound $u^t-\fint u^t$. From \eqref{eq:BBestimates} and \eqref{eq:gestimates}, we bound for any $r\in [1,\infty]$ and $s\in (1,\infty)$
\begin{align*}
\norm{w^t}_{L^r_tL_x^s}\leq& C(\mathring{R}_0,\eps ,s )\kappa^{1-\frac{1}{r}}(\mu_1^{1-\frac{1}{s}}\mu_2^{1-\frac{1}{s}}\omega^{-1}+\omega^{-1}),\\
\norm{\na w^t}_{L^r_tL_x^s}\leq& C(\mathring{R}_0,\eps,s)\kappa^{1-\frac{1}{r}}(\lambda\mu_1^{1-\frac{1}{s}}\mu_2^{2-\frac{1}{s}}\omega^{-1}+\omega^{-1}).
\end{align*}
Finally, we estimate the time-intermittency corrector $w^g$. Since $L^1(\T)\subset L^p(\T)$ continuously for any $p>1$, we bound for any $r\in [1,\infty]$ and $s\in [1,\infty)$
\begin{equation}\label{eq:wgbounds}
\begin{aligned}
\norm{w^g}_{L_{t}^rL^s_x}&\leq C(\mathring{R}_0,s)\nu^{-1}, \\
\norm{\na w^g}_{L_{t}^rL^s_x}&\leq C(\mathring{R}_0,s)\nu^{-1}.  
\end{aligned}
\end{equation}

\subsubsection{Estimates of the gradient in Hardy spaces}\label{sec:estimatescurl}
In this subsection we estimate the gradient of the perturbation $w$ in $\Hp$. 

The spatial structure of $w^p$ is analogous to the structure from \cite{buckmodena2024compactly}, but we have a multiplicative factor $g_\kappa$, and we work on $\T^2$ instead of $\R^2$, which requires certain modifications. Therefore, we include the proof of the estimates. 

By \eqref{eq:wp}, we have 
$$\na w^p=g_{\kappa}(\nu t)\sum_k \na \nabla^\perp \curl \Delta^N [\divergence(a_k(A_k+A_k^T))].$$
From \eqref{eq:BBdef}, we observe that $a_k(A_k+A_k^T)$ is supported on $\lambda^2$-many small, disjoint balls of radius $\lambda^{-1}\mu_1^{-1}$. 
Since there is a derivative of order $2N+4$ acting on it, and recalling Definition \ref{def:mainatom}, we can apply Corollary \ref{cor:mainLinftyinHp} to each of these disjoint parts. Then, Proposition \ref{prop:mainHpcomplete} gives

\begin{align*}
\norm{\na w^p(t)}_{\Hp_x}^p\lesssim \frac{\lambda^2}{\lambda^2\mu_1^2}\norm{g_\kappa a_k(A_k+A_k^T)}_{W^{2N+4,\infty}}^p\lesssim C(\mathring{R}_0,\eps  )\frac{1}{\mu_1^2}\kappa^{\frac{p}{2}}\lambda^p \mu_1^{\frac{p}{2}}\mu_2^\frac{3p}{2},
\end{align*}
where we have used \eqref{eq:gestimates}, \eqref{eq:abound} and \eqref{eq:BBestimates}.
Thus, we have that $\na w^p\in \Hp$ and 
\begin{equation}\label{eq:nawpHp}
\norm{\na w^p}_{C_t\Hp_x}\leq C(\mathring{R}_0,\eps  )\kappa^{\frac{1}{2}}\lambda \mu_1^{\frac{1}{2}-\frac{2}{p}}\mu_2^\frac{3}{2}.    
\end{equation}
To control $\na w^t$, we use Corollary \ref{prop:mainPboundedHp} to deduce
$$\norm{\na w^t}_{C_t\Hp_x}\lesssim \norm{\na u^t}_{C_t\Hp_x}\lesssim \norm{g_\kappa^2(Y_k-\frac{\xi_k}{\omega })\na (a_k^2)}_{C_t\Hp_x}+\norm{g_\kappa^2 a_k^2 \na Y_k}_{C_t\Hp_x}.$$ 
For the first term, we use Corollary \ref{cor:mainhardyembedding} to bound 
$$\norm{g_\kappa^2(Y_k-\frac{\xi_k}{\omega })\na (a_k^2)}_{C_t\Hp_x}\leq \norm{g_\kappa^2(Y_k-\frac{\xi_k}{\omega })\na (a_k^2)}_{C_t L^2_x}\leq C(\mathring{R}_0,\eps )\kappa\omega^{-1}\mu_1^{\frac{1}{2}}\mu_2^{\frac{1}{2}}.$$
To bound the second term, we note that we can write  $g_\kappa^2a_k^2\na Y_k=\sum_{i=1}^{\lambda^2}\theta_{i}$, with each $\theta_{i}$ supported in a ball of radius $\lambda^{-1}\mu_1^{-1}$, with balls mutually disjoint. We denote by $\bar{B}_{\theta_i}\subset \R^n$ the representative ball associated to $\theta_i$, as in Proposition \ref{prop:mainHpboundnonvanishingatoms}. Using the estimates for building blocks \eqref{eq:BBestimates} and Proposition \ref{prop:mainHpboundnonvanishingatoms}, we deduce 
$$\norm{\theta_i}^p_{C_t\Hp_x}\leq C(\mathring{R}_0,\eps )\big(\kappa^p\lambda^{-2+p}\mu_1^{-2+p}\mu_2^{2p}\omega^{-p}+|\log{(\lambda \mu_{1})}|\max_{|\alpha|\leq N}\big|\int_{\bar{B}_{\theta_i}} x^\alpha \theta_i(x) dx \big|^p\big).$$
Since the supports of $\theta_i$ are disjoint, we can write
$$\int_{\bar{B}_{\theta_i}} x^\alpha \theta_i(x) dx=-\int_{\bar{B}_{\theta_i}} \na(x^\alpha g_\kappa^2a_k^2 )Y_k dx,$$
which implies 
$$\max_{|\alpha|\leq N}\big|\int_{\bar{B}_{\theta_i}} x^\alpha \theta_i(x) dx \big|^p\leq C(\mathring{R}_0,\eps )\kappa^p\norm{Y_k}_{L^1}^p.$$
We use \eqref{eq:BBestimates} to deduce 
$$\norm{\theta_i}^p_{C_t\Hp_x}\leq C(\mathring{R}_0,\eps )\big(\kappa^p\lambda^{-2+p}\mu_1^{-2+p}\mu_2^{2p}\omega^{-p}+|\log{(\lambda \mu_{1})}|\kappa^p\omega^{-p}\big),$$
which implies
$$\norm{g_\kappa^2a_k^2\na Y_k}^p_{C_t\Hp_x}\lesssim C(\mathring{R}_0,\eps, p)\big(\kappa^p\lambda^{p}\mu_1^{-2+p}\mu_2^{2p}\omega^{-p}+|\log{(\lambda \mu_{1})}|\kappa^p\omega^{-p}\lambda^{2}\big),$$
so 
\begin{align*}
\norm{\na w^t}_{C_t\Hp_x}\leq C(\mathring{R}_0,\eps )\big(\kappa\omega^{-1}\mu_1^{\frac{1}{2}}\mu_2^{\frac{1}{2}}+\kappa\lambda\mu_1^{-\frac{2}{p}+1}\mu_2^{2}\omega^{-1}+|\log{(\lambda \mu_{1})}|^{\frac{1}{p}}\kappa\omega^{-1}\lambda^{\frac{2}{p}}\big).
\end{align*}
Finally, Remark \ref{cor:mainLinftyinHp} and \eqref{eq:wgbounds} gives 
$$\norm{\nabla w^{g}}_{C_t\Hp}\leq C(\mathring{R}_0)\nu^{-1}.$$

\section{The new error}\label{sec:newerrorR}
In this section we define and estimate $R_1$. Recall that $u_0$ solves \eqref{eq:NSR} with $R_0$, and that $u_1=u_0+w$. Therefore
\begin{equation}\label{eq:newR1p1}
\begin{aligned}
-\divergence\mathring{R}_1=&\divergence(u_0\otimes w+w\otimes u_0) & =:\divergence R^{\textnormal{lin},1}\\
&+\divergence(u^p\otimes(w-u^p))+\divergence((w-u^p)\otimes u^p)&=:\divergence R^{\textnormal{lin},2}\\
&+\divergence((w-u^p)\otimes (w-u^p))&=:\divergence R^{\textnormal{lin},3}\\
&+\divergence(\na w {+(\na w)^T}) &=:\divergence R^{\Delta}\\
&+\pa_t w+\divergence(u^p\otimes u^p-\mathring{R}_0) &=:E^Q\\
&+\na(p_1-\tilde p_0).
\end{aligned}
\end{equation}

\subsection{Linear errors $R^{\textnormal{lin}}$}
Using estimates from section \ref{sec:estimatesperturbations} we have, since $w = u^p+ u^c+w^t+w^{g}$
$$
\begin{aligned}
\norm{R^{\text{lin},1}}_{L^1_{t,x}}\leq &2\norm{u_0}_{C_t L^2_x}(\norm{u^c}_{L^1_tL^2_x}+\norm{w^{t}}_{L^1_tL^2_x}+\norm{w^{g}}_{L^1_tL^2_x})+2\norm{u_0}_{C_t L^\infty_x}\norm{u^p}_{L^1_tL^1_x}\\
\leq& C(\mathring{R}_0,\eps )\left(\kappa^{-\frac{1}{2}}\frac{\mu_1}{\mu_2}+\frac{\mu_1^\frac{1}{2}\mu_2^\frac{1}{2}}{\omega} +\omega^{-1}+\nu^{-1}+\kappa^{-\frac{1}{2}}\mu_1^{-\frac{1}{2}}\mu_2^{-\frac{1}{2}}\right), \\
\norm{R^{\text{lin},2}}_{L^1_{t,x}}\leq &2\norm{u_p}_{L^2_t L^2_x}(\norm{u^c}_{L^2_tL^2_x}+\norm{w^{t}}_{L^2_tL^2_x}+\norm{w^{g}}_{L^2_tL^2_x})\\
\leq& C(\mathring{R}_0,\eps )\left(\frac{\mu_1}{\mu_2}+\frac{\kappa^\frac{1}{2}\mu_1^\frac{1}{2}\mu_2^\frac{1}{2}}{\omega} +\kappa^\frac{1}{2}\omega^{-1}+\nu^{-1}\right),\\
\norm{R^{\text{lin},3}}_{L^1_{t,x}}&\le C \norm{u^c}_{L^2_tL^2_x}^2+\norm{w^{t}}_{L^2_tL^2_x}^2+\norm{w^{g}}_{L^2_tL^2_x}^2\\
&\leq C(\mathring{R}_0,\eps )\left(\frac{\mu_1^2}{\mu_2^2}+\frac{\kappa\mu_1\mu_2}{\omega^2}+\kappa\omega^{-2}+\nu^{-2}\right),
\end{aligned}
$$

\subsection{Dissipative error $R^{\Delta}$}
The estimates from section \ref{sec:estimatesperturbations} give, for any $1<s<\infty$,
$$
\norm{R^{\Delta}}_{L^1_{t,x}}\leq C(\mathring{R}_0,\eps ,s)\Big(\kappa^{-\frac{1}{2}}\lambda\mu_1^{-\frac{1}{2}}\mu_2^\frac{1}{2}
+\lambda\mu_1^{1-\frac{1}{s}}\mu_2^{2-\frac{1}{s}}\omega^{-1}+\omega^{-1}+\nu^{-1}\Big).
$$

\subsection{Error $E^{Q}$} 
Let us first write
\[
\mathring{R}_0 = g^2_\kappa(\nu\cdot) \mathring{R}_0 + (1-g^2_\kappa(\nu \cdot) )\mathring{R}_0 .
\]
Using this, now we split $E^{Q}$ of \eqref{eq:newR1p1}
as follows
\begin{equation}\label{eq:balanceerrorR0}
\begin{aligned}
E^{Q}=&\pa_t(u^t-\fint u^t)+\divergence(u^p\otimes u^p-g^2_\kappa(\nu\cdot) \mathring{R}_0)\\
&+\pa_t u^{g} + (1-g^2_\kappa(\nu \cdot))\divergence(\mathring{R}_0)\\
&+\pa_t w^p\\
&-\pa_t \na \Delta^{-1}\divergence u^t+\pa_t u^{ccc}.
\end{aligned}
\end{equation}
Let us immediately observe, in view of Proposition \ref{prop:defmaincomponentsandcorrectors}, that the last line of \eqref{eq:balanceerrorR0} contains terms which are gradients, so we can define
$$\na\pi_1=-\na \pa_t\Delta^{-1}\divergence u^t+\pa_t u^{ccc}.$$
This term can be included into the new pressure, so we don't need to estimate it. 
\subsubsection{R.h.s. of the first line of \eqref{eq:balanceerrorR0}}
Using the properties of the building blocks from section \ref{sssec:sp} and the definition \eqref{eq:up}of $u^p$, we have for the r.h.s. of the first line of \eqref{eq:balanceerrorR0}
\begin{align*}
\pa_t(u^t-\fint u^t)+\divergence(u^p\otimes u^p-g^2_\kappa(\nu\cdot) \mathring{R}_0)= T -\pa_t\fint u^t+ \nabla \pi_2
\end{align*}
where 
\begin{align*}
T &:= -\sum_k\pa_t (g^2_\kappa a_k^2) \left(Y_k-\frac{1}{\omega}\xi_k\right) 
+g_\kappa^2\sum_k \left(W_k\otimes W_k-\frac{\xi_k}{|\xi_k|}\otimes \frac{\xi_k}{|\xi_k|} \right)\cdot \na a_k^2,\\
\nabla \pi_2 &:= \divergence g^2_\kappa (\rho I) = \nabla ( g^2_\kappa \rho).
\end{align*} The aim of our construction was to remove fast means, as apparent above. Thus, thanks to Lemma \ref{lem:antidivergenceoperators}, we can invert divergence. Therefore
\begin{align*}
T-\fint T= \divergence R^Y +  \divergence R^Q,
\end{align*}
where
\begin{align*}
R^{Y}=&-\sum_k \fancyR\left(\pa_t (g_\kappa^2 a_k^2), Y_k-\frac{1}{\omega}\xi_k\right),\\
R^{Q}=&g_\kappa^2\sum_k\tilde{\fancyR}\left(\na a_k^2, W_k\otimes W_k-\frac{\xi_k}{|\xi_k|}\otimes \frac{\xi_k}{|\xi_k|}\right).
\end{align*}
Using Lemma \ref{lem:antidivergenceoperators} and building blocks estimates \eqref{eq:BBestimates}, we see that
\begin{align*}
\norm{R^{Y}}_{L^1_{t,x}}\leq& \frac{C(\mathring{R}_0,\eps,\tau  )\nu \kappa}{\omega \lambda},\\
\norm{R^{Q}}_{L^1_{t,x}}\leq& \frac{C(\mathring{R}_0,\eps)}{\lambda}.
\end{align*}
\subsubsection{Second line of \eqref{eq:balanceerrorR0}}
Using \eqref{eq:h} and \eqref{eq:ug} we have 
\[
\pa_t u^{g} + (1-g^2_\kappa(\nu \cdot))\divergence(\mathring{R_0} ) =-\frac{h_\kappa}{\nu}\divergence(\pa_t\mathring{R}_0)=\divergence R^g,
\]
where we have defined 
$$R^g=-\frac{h_\kappa}{\nu}\pa_t\mathring{R}_0.$$
Since $h_\kappa$ is bounded, we see that
$$\norm{R^{g}}_{L^1_{t,x}}\leq \frac{C(\mathring{R}_0)}{\nu}.$$

\subsubsection{Third line of \eqref{eq:balanceerrorR0}}
We use formula \eqref{eq:wp} and the identity $\na^\perp\curl=\Delta-\na\divergence$ to obtain, since $\divergence$ and $\Delta$ commute, 
\begin{align*}
\pa_t w^p=&\sum_k \divergence (\Delta^{N+1}(\nu\kappa g'_\kappa a_\kappa(A_k+A_k^T)+g_\kappa\pa_t (a_\kappa(A_k+A_k^T)))) \\
&-\na \divergence \Delta^N\divergence\pa_t(g_\kappa a_\kappa(A_k+A_k^T))\coloneqq \divergence R^{\text{time}}+\na \pi_3,
\end{align*}
where 
$$R^{\text{time}}=\sum_k \Delta^{N+1}(\nu\kappa g'_\kappa a_\kappa(A_k+A_k^T)+g_\kappa\pa_t (a_\kappa(A_k+A_k^T))).$$
Using \eqref{eq:BBestimates} and \eqref{eq:gestimates}, we bound
\begin{equation}
\norm{R^{\text{time}}}_{L^1_{t,x}}\leq C(\mathring{R}_0,\eps,\tau  )\Big( \frac{\nu \kappa^\frac{1}{2}}{\lambda\mu_1^\frac{1}{2}\mu_2^\frac{3}{2}} +  \frac{\omega \mu_1^\frac{1}{2}}{\kappa^{\frac{1}{2}}\mu_2^\frac{3}{2}}\Big).
\end{equation}
\subsubsection{Final formula for error $E^{Q}$ }
Recall \eqref{eq:balanceerrorR0}. Putting together all the considerations concerning its r.h.s., we have
\begin{equation}\label{eq:EQfinalformula}
\begin{aligned}
E^{Q} 
=&\divergence (R^{Y}+R^{Q}+
R^g+R^{\text{time}})\\
&+\na (\pi_2+\pi_3)+\fint T-\fint \pa_t  u_t.
\end{aligned}
\end{equation}
We note that $E^Q$ has zero mean (in view of \eqref{eq:newR1p1}, using the fact that the perturbation $w$ has zero mean). Thus, integrating equation \eqref{eq:EQfinalformula}, we see that 
$$\fint T-\fint \pa_t  u_t=0,$$
so 
\begin{equation*}
\begin{aligned}
E^{Q} 
=&\divergence (R^{Y}+R^{Q}+
R^g+R^{\text{time}})+\na (\pi_2+\pi_3).
\end{aligned}
\end{equation*}

To conclude the section, we define
\begin{align*}
R_1=&-(R^{\text{lin,1}}+R^{\text{lin,2}}+R^{\text{lin,3}}+R^{\Delta}+R^{Y}+R^{Q}+R^g+R^{\text{time}}),\\
\na \pi_1=&\na \tilde{\pi}_0+\na\pi_1+\na\pi_2+\na\pi_3+\frac{1}{2}\na \text{tr}R_1.
\end{align*}
With this, $(u_1,\pi_1,R_1)$ is a solution to the NS-Reynolds \eqref{eq:NSR}.

\section{Conclusion of proof of main proposition}
Having constructed the new solution to \eqref{eq:NSR} $(u_1,\pi_1,R_1)$ in the previous sections, to prove Proposition \ref{prop:main} it remains to show that the new error is small.

Recall that the Proposition fixes $\sigma\in(0,1)$ (of $W^{\sigma,1}$), $p \in (0,1)$ (of $\Hp$), $\tau>0$ (parameter capturing increase in time support), and $\delta>0$ (target smallness).

Our aim now is to choose the parameters $\eps, \mu_1,\mu_2,\kappa, \omega, \nu, \lambda, s $ so that \eqref{eq:mainpropR},  \eqref{eq:mainpropu},  \eqref{eq:mainpropcurl} {and \eqref{eq:mainpropCL}} of
Proposition \ref{prop:main} hold. We will base the parameters on $\lambda$, i.e. let
$$\begin{aligned}
\mu_1=&\lambda^\alpha,\\
\mu_2=&\lambda^{\alpha+a},\\
\kappa^{1/2}=&\lambda^\beta,\\
\omega=&\lambda^{\alpha+b}, \\
\nu=& \lambda^{\gamma},
\end{aligned}$$
and choose 
\[
\eps=\norm{\mathring{R}_0}_{L^1_{t,x}} \wedge \tau.
\]
It will suffice to take\footnote{There is certain flexibility in choices of the parameters, we provide one concrete choice for clarity.}
\begin{align*}
\beta=4,\,
a=5,\,
b=11,\,
\gamma=1,
\end{align*}
$\alpha$ big enough so that
\begin{align*}
\alpha>&\frac{2}{p},\\
2\alpha\left(1-\frac{1}{p}\right)<&-\frac{25}{2},\\
\alpha(1-\sigma)>&\beta+\sigma+a(\sigma-\frac{1}{2}),\\
\alpha\frac{1-\sigma^2}{3+\sigma}>&2\beta+\sigma+a(\frac{1-\sigma}{3+\sigma}+\sigma)-b,
\end{align*}
and $s>1$ close to one, so that 
$$-5+(5+2\alpha)\big(1-\frac{1}{s}\big)<0. $$
With these choices, we will show that all errors terms are bounded by negative powers of $\lambda$ times a constant, which depends (if it is not the explicit $M$) only on $\mathring{R}_0, \varepsilon$. Thus, taking $\lambda\gg 1$ will make the errors as small as we want, yielding {\eqref{eq:mainsubequations}}. 
The $L_{x,t}^2$ error estimate will require precise control over its constant.

\subsection{Bound of $\norm{u_1-u_0}_{L^2_{t,x}}$}
Using the estimates from Section \ref{subsec:estimateslebesgue}, we bound
\begin{equation*}
\begin{aligned}
\norm{u_1-u_0}_{L^2_{t,x}}=& \norm{u^p+u^c+w^t+w^g}_{L^2_{t,x}} \\
\leq& 4\sqrt{10}(\norm{\mathring{R}_0}^\frac{1}{2}_{L^1_{t,x}}+\eps^\frac{1}{2})+C(\mathring{R}_0,\eps )(\lambda^{-\frac{1}{2}}+\mu_1\mu_2^{-1}+\kappa^{\frac{1}{2}} \mu_1^\frac{1}{2}\mu_2^\frac{1}{2}\omega^{-1}+\omega^{-1}+\nu^{-1})\\
\le & 8\sqrt{10}\norm{\mathring{R}_0}^\frac{1}{2}_{L^1_{t,x}}+C(\mathring{R}_0)(\lambda^{-\frac{1}{2}}+\lambda^{-5}+\lambda^{-\frac{9}{2}}+\lambda^{-(\alpha+11)}+\lambda^{-1}).
\end{aligned}
\end{equation*}
Thus for  $\lambda\gg 1$, it holds
$$\norm{u_1-u_0}_{L^2_{t,x}}\le M \norm{\mathring{R}_0}^\frac{1}{2}_{L^1_{t,x}} + \delta,$$
with $M = 8\sqrt{10}$,
giving \eqref{eq:mainpropu}.

\subsection{Bound of $\norm{\na{(u_1-u_0)}}_{C_t\Hp_x}$}

Using the estimates from Section \ref{sec:estimatescurl}, we bound
\begin{align*}
\norm{\na{(u_1-u_0)}}_{\Hp_x}=& 
\norm{w^p+w^t+w^g}\\
\lesssim&
\kappa^\frac{1}{2}\lambda\mu_1^{\frac{1}{2}-\frac{2}{p}}\mu_2^\frac{3}{2}+\kappa\omega^{-1}\mu_1^{\frac{1}{2}}\mu_2^{\frac{1}{2}}+\kappa\lambda\mu_1^{-\frac{2}{p}+1}\mu_2^{2}\omega^{-1}\\
&+|\log{(\lambda \mu_{1})}|^{\frac{1}{p}}\kappa\omega^{-1}\lambda^{\frac{2}{p}}+\nu^{-1}\\
=& \lambda^{2\alpha(1-\frac{1}{p})+\frac{25}{2}}+\lambda^{-\frac{1}{2}}+\lambda^{2\alpha(1-\frac{1}{p})+8}+(\alpha+1)^\frac{1}{p}\log(\lambda)^\frac{1}{p}\lambda^{-\alpha-3+\frac{2}{p}}+\lambda^{-1},
\end{align*}
uniformly in $t$. 
Thus, taking $\lambda\gg 1$ gives $\norm{\na{(u_1-u_0)}}_{C_t\Hp_x}<\delta$.

\subsection{Bound of $\norm{u_1-u_0}_{C_tW^{\sigma,1}_{x}}$}
We fix $0<\sigma<1$. Using the estimates from Section \ref{subsec:estimateslebesgue} and that $\mu_2\gg \mu_1$, we obtain
\begin{equation*}
\begin{aligned}
\norm{u_1-u_0}_{W^{\sigma,1}_{x}}=& \norm{u_1-u_0}^{1-\sigma}_{L^1_x}\norm{\na(u_1-u_0)}^{\sigma}_{L^1_x} \\
\leq& \norm{w^p}^{1-\sigma}_{L^1_x}\norm{\na w^p}^{\sigma}_{L^1_x}+\norm{w^t}^{1-\sigma}_{L^\frac{3+\sigma}{2+2\sigma}_x}\norm{\na w^t }^{\sigma}_{L^\frac{3+\sigma}{2+2\sigma}_x}+\norm{w^g}^{1-\sigma}_{L^1_x}\norm{\na w^g }^{\sigma}_{L^1_x} \\
\lesssim & \kappa^\frac{1}{2}\lambda^\sigma\mu_1^{-\frac{1}{2}}\mu_2^{-\frac{1}{2}+\sigma}+
\kappa\omega^{-1}+\kappa\lambda^\sigma\mu_1^{\frac{1-\sigma}{3+\sigma}}\mu_2^{\frac{1-\sigma}{3+\sigma}+\sigma}\omega^{-1}+\nu^{-1}\\
=& \lambda^{-\alpha(1-\sigma)+\beta+\sigma+a(\sigma-\frac{1}{2})}
+\lambda^{2\beta-\alpha-b}
+\lambda^{-\alpha\frac{1-\sigma^2}{3+\sigma}+2\beta+\sigma+a(\frac{1-\sigma}{3+\sigma}+\sigma)-b}
+\lambda^{-\gamma},
\end{aligned}
\end{equation*}
uniformly in $t$. Thus, taking $\lambda\gg 1$ gives $\norm{u_1-u_0}_{C_tW^{\sigma,1}_x}<\delta$.

\color{black}

\subsection{Bound of $\norm{R_1}_{L^1_{t,x}}$}
Now we use the estimates from Section \ref{sec:newerrorR}. We bound each component of $R_1=-(R^{\text{lin,1}}+R^{\text{lin,2}}+R^{\text{lin,3}}+R^{\Delta}+R^{Y}+R^{Q}+R^g+R^{\text{time}})$ separately.

\begin{align*}
\norm{R^{\text{lin,1}}}_{L^1_{t,x}}\lesssim &\kappa^{-\frac{1}{2}}\frac{\mu_1}{\mu_2}+\frac{\mu_1^\frac{1}{2}\mu_2^\frac{1}{2}}{\omega} +\omega^{-1}+\nu^{-1}+\kappa^{-\frac{1}{2}}\mu_1^{-\frac{1}{2}}\mu_2^{-\frac{1}{2}}\\
=&\lambda^{-9}+\lambda^{-\frac{17}{2}}+\lambda^{-(\alpha+11)}+\lambda^{-1}+\lambda^{-(\alpha+\frac{13}{2})},    \\
\norm{R^{\text{lin,2}}}_{L^1_{t,x}}\lesssim &\frac{\mu_1}{\mu_2}+\frac{\kappa^\frac{1}{2}\mu_1^\frac{1}{2}\mu_2^\frac{1}{2}}{\omega} +\kappa^\frac{1}{2}\omega^{-1}+\nu^{-1}=\lambda^{-5}+\lambda^{-\frac{9}{2}}+\lambda^{-(\alpha+7)}+\lambda^{-1},\\
\norm{R^{\text{lin,3}}}_{L^1_{t,x}}\lesssim &\frac{\mu_1^2}{\mu_2^2}+\frac{\kappa\mu_1\mu_2}{\omega^2}+\kappa\omega^{-2}+\nu^{-2}=\lambda^{-10}+\lambda^{-9}+\lambda^{-2(\alpha+7)}+\lambda^{-2},\\
\norm{R^{\Delta}}_{L^1_{t,x}}\lesssim &\kappa^{-\frac{1}{2}}\lambda\mu_1^{-\frac{1}{2}}\mu_2^\frac{1}{2}
+\lambda\mu_1^{1-\frac{1}{s}}\mu_2^{2-\frac{1}{s}}\omega^{-1}+\omega^{-1}+\nu^{-1}\\
=&\lambda^{-\frac{1}{2}}+\lambda^{-5+(5+2\alpha)(1-\frac{1}{s})}+\lambda^{-(\alpha+11)}+\lambda^{-1},\\
\norm{R^{Y}}_{L^1_{t,x}}\lesssim & \frac{\nu \kappa}{\omega \lambda}=\lambda^{-(\alpha+3)},\\
\norm{R^{Q}}_{L^1_{t,x}}\lesssim &\lambda^{-1},\\
\norm{R^{g}}_{L^1_{t,x}}\lesssim &\lambda^{-1},\\
\norm{R^{\text{time}}}_{L^1_{t,x}}\lesssim & \frac{\nu \kappa^\frac{1}{2}}{\lambda\mu_1^\frac{1}{2}\mu_2^\frac{3}{2}} +  \frac{\omega \mu_1^\frac{1}{2}}{\kappa^{\frac{1}{2}}\mu_2^\frac{3}{2}}=\lambda^{-2\alpha-\frac{7}{2}}+\lambda^{-\frac{1}{2}}.
\end{align*}
All the exponents of $\lambda$ are strictly negative, so $\lambda\gg 1$ yields $\norm{R_1}_{L^1_{t,x}}<\delta$, concluding the proof of estimates in Proposition \ref{prop:main}.

\appendix
\section{Real Hardy spaces $\Hp(\T^n$)}\label{sec:Hardy}
Since we could not find a suitable reference for Hardy spaces on $\T^n$, we develop  the relevant theory.
The proofs ideas are mainly adapted from \cite{follandsteinhardy,stein1993harmonic}.

The properties of $\Hp(\T^n)$ are analogue to the properties of $\Hp(\R^n)$, but with several useful advantages of the periodic (compact) setting over the full-space case. These are:
\begin{enumerate}
    \item the embedding $\Hp(\T^n)\subset \mathcal{H}^q(\T^n)$ for any $0<q<p\leq \infty$ (as it occurs for Lebesgue spaces),
    \item the fact that not all atoms must have vanishing moments, cf. Definition \ref{def:atoms} (expected in view of the analogy between $\Hp(\T^n)$ with $\Hp_{loc}(\R^n)$, defined in \cite{goldberg1979}).
\end{enumerate}
From now on, if the domain is not stated explicitly, it is $\T^n$, Also, to simplify the exposition, we fix $\T^n=[-1/2,1/2]^n$ with periodic boundary conditions.

\subsection{Definition of $\Hp(\T^n)$}
We will define Hardy spaces using maximal functions. To this end, one introduces a rescaling parameter for a test function, which is more natural in $\R^n$. We understand the convolution between a function $\Psi \in \fancyS(\R^n)$ and $f\in \fancyD'(\T^n)$ as the convolution in $\R^n$ of $\Psi$ and ${f^{\text{ext}}}$, which is the periodic extension of $f$. To avoid confusion, we denote by $*_{\T^n}$ the convolution between periodic functions, by $*_{\R^n}$ the convolution between functions on $\R^n$, and by $*$ the above mentioned convolution of $\Psi \in \fancyS(\R^n)$ and $f\in \fancyD'(\T^n)$, i.e.\ $\Psi*f := \Psi*_{\R^n} {f^\text{ext}}$.

Since the resulting function $\Psi*f$ inherits the period of $f$, we treat it as a periodic function (without renaming).

\begin{defi}[Maximal function]\label{def:mf}
Let $\Psi\in\fancyS(\R^n)$ such that $\int_{\T^n}\Psi(x)dx\neq0$. For each $0<\zeta< \infty$, let ${\Psi}_\zeta(x)=\zeta^{-n}\Psi(\frac{x}{\zeta})$. Given $f\in \fancyD'(\T^n)$, we define the maximal function
$$m_\Psi f(x)=\sup_{0<\zeta<\infty}|f*\Psi_\zeta(x)|.$$
\end{defi}

\begin{defi}[Hardy space]
Let $0<p<\infty$. The real Hardy space $\Hp$ is defined as
$$\Hp=\{f\in\fancyD'(\T^n):m_\Psi f\in L^p(\T^n) \},$$
and we write
$$\norm{f}_{\Hp}=\norm{m_\Psi f}_{L^p(\T^n)}.$$
\end{defi}Observe that, since in case $p \in (0,1)$ one has merely $(a+b)^p \le 2^{p-1}(a^p + b^p)$, $\norm{f}_{\Hp}$ is only a quasinorm. 

\subsubsection{Equivalent definitions via other maximal functions}
Let us discuss two other ways to introduce Hardy spaces, based on two other notions of a maximal function. The first one uses the non-centered\footnote{This maximal function is usually called non-tangential maximal function \cite{stein1993harmonic}. We prefer to use non-centered instead to avoid possible confusion.} maximal function:
\begin{equation}\label{eq:nontanmaxfunction}
m_{\Psi}^*f(x)=\sup_{|x-y|<\zeta}|(f*\Psi_\zeta)(y)|.   
\end{equation}

To arrive at the second way, we recall first that $\fancyS(\R^n)$ is equipped with the family of seminorms $$\norm{\cdot}_{\alpha,\beta}=\sup_{x\in\R^n}|x^\alpha\na^\beta(\cdot)|,$$ where $\alpha,\beta\in \N$. We can fix a set $\fancyF$ of seminorms from $\fancyS(\R^n)$, and consequently the set 
$$\fancyS_\fancyF=\{\Psi\in \fancyS(\R^n):\norm{\Psi}_{\alpha_i,\beta_i}\leq 1 \text{ for all } \norm{\cdot}_{\alpha_i,\beta_i}\in \fancyF\}.$$
Then, we define the grand maximal function based on $\fancyF$ as 
$$m_\fancyF f(x)=\sup_{\Psi\in\fancyS_\fancyF}|m_\Psi f(x)|.$$
It turns out that each of these notions is, under mild assumptions, equivalent to our original definition, in view of

\begin{lemma}\label{lem:hardycomparableseminorms}
Let $f\in \fancyD'(\T^n)$. For each $N\in \N$, let $\fancyF_N$ be the set of seminorms
$$\fancyF_N=\{\norm{\cdot}_{\alpha,\beta}:|\alpha|\leq N,|\beta|\leq N\}.$$
Then, if $N$ is sufficiently large in terms of $p$ and $n$, the quantities
$$\norm{m_\Psi f}_{L^p}\text{, } \norm{m_\Psi^* f}_{L^p} \text{ and }\norm{m_{\fancyF_N} f}_{L^p}$$
are mutually comparable, with bounds independent of $f$. The implicit bounds may depend, however, on $\Psi, p, n$. The dependence on $\Psi$ can be uniformly bounded in $\{\Psi'\in \fancyS(\R^n):\norm{\Psi'}_\fancyF\leq C\}$, for a certain family of seminorms $\fancyF$.

In particular, the fact that $f\in \Hp$ does not depend on the particular choice of $\Psi$ in the Definition \ref{def:mf} of the maximal function. 
\end{lemma}
For a proof of this Lemma in $\Hp(\R^n)$ that also works in $\Hp$, see \cite[Section III.1]{stein1993harmonic}. Note that we work with the same space of test functions, and the subsequent computations differ only slightly due to the change of domain.
The precise minimum value of $N$ is given in \cite[p. 72]{follandsteinhardy}, and it is $N(p)=\lfloor n(\frac{1}{p}-1)\rfloor+1$.

\subsection{Comparison with $L^p$ and basic properties}
For any $0<p\leq q\leq\infty$, Hölder inequality gives
\begin{equation}\label{eq:hardyembedding}
\norm{f}_{\Hp}\leq \norm{f}_{\mathcal{H}^q}.
\end{equation} 
In fact, there is nothing new in the case $p>1$ when one compares Hardy and Lebesgue spaces, in view of 

\begin{lemma}
For $1<p\leq \infty$, $\Hp\equiv L^p (\T^n)$.
\end{lemma}The proof of the above Lemma follows the same lines as \cite[Chapter III, Remark 1.2.2.]{stein1993harmonic}.

The advantages of using Hardy spaces, mentioned in the introduction, appear in the range $0<p\leq 1$.
To define $\Hp$ we have used a test function $\Psi\in \fancyS(\R^n)$. 

We intend to show now that $\Hp\subset \fancyD'(\T^n)$. To this end, we need to know first how to appropriately convert periodic test functions into full space ones, and vice versa: \\
(i) Given $\Psi\in \fancyS(\R^n)$, one can easily construct a periodic test function ${\Psi^\text{per}}\in \fancyD(\T^n)$ as
$$\Psi^\text{per}(\cdot)\coloneqq\sum_{i\in\Z^n}\Psi(\cdot+i).$$
This gives, for any $f\in \fancyD'(\T^n)$, 
$$\Psi*f=\Psi^\text{per}*_{\T^n}f,$$
\\
(ii) Conversely, we can also construct a test function in the whole space from a periodic one. Let $\Xi\in C^\infty_c (\R^n)$ be a positive function such that $\Xi\equiv 1$ on $B(0,2)$, and define
$$\tilde{\Xi}(\cdot)\coloneqq\frac{\Xi(\cdot)}{\sum_{i\in\Z^n}\Xi(\cdot+i)}.$$
Given $\tilde{\Psi}\in \fancyD(\T^n)$, and viewing it as a periodic function defined on $\R^n$, the function 
\begin{equation}\label{e:rule_TtoR}
\tilde{\Psi}^\text{exp}(\cdot)\coloneqq \tilde{\Psi}(\cdot)\tilde{\Xi}(\cdot)
\end{equation}
belongs to $\fancyS(\R^n)$. $\tilde{\Psi}^\text{exp}$ is an expanded (in $\R^n$) version of $\tilde{\Psi}$. In particular, for any $f\in \fancyD'(\T^n)$,

\begin{equation}\label{e:rule_TtoR_id}
\tilde{\Psi}*_{\T^n}f=\tilde{\Psi}^\text{exp}*f.
\end{equation}

Denoting by $\hookrightarrow$ a continuous inclusion, we have
\begin{prop}
For $0<p\leq 1$, $\Hp\hookrightarrow \fancyD'(\T^n)$.
\end{prop}
\bpf
Let $\tilde{\phi}\in \fancyD(\T^n)$, and define $\underline{\tilde{\phi}}(x)=\tilde{\phi}(-x)$. We denote by ${\underline{\tilde{\phi}}}^\text{exp}$ the function obtained from $\underline{\tilde{\phi}}$ according to the rule \eqref{e:rule_TtoR}. Via \eqref{e:rule_TtoR_id} and by definition \eqref{eq:nontanmaxfunction} we have
$$|\int_{\T^n} f(x)\tilde{\phi}(x)dx|=|f*_{\T^n} \tilde{\underline{\phi}}(0)|=|f*\underline{\tilde{\phi}}^\text{exp}(0)|\leq m_{\underline{\tilde{\phi}}^\text{exp}}^*f(y)\quad \forall |y|\leq 1.$$
Therefore, by Lemma \ref{lem:hardycomparableseminorms},
\begin{align*}
|\int_{\T^n} f(x)\tilde{\phi}(x)dx|^p\lesssim & \int_{\T^n}(m_{\underline{\tilde{\phi}}^\text{exp}}^*(f(y))^pdy\lesssim C(\underline{\tilde{\phi}}^\text{exp})\norm{\underline{\tilde{\phi}}^\text{exp}}_{\fancyF_N}^p\int_{\T^n}(m_{\fancyF_N}f(y))^pdy \\
\lesssim &C(\underline{\tilde{\phi}}^\text{exp})\norm{\underline{\tilde{\phi}}^\text{exp}}_{\fancyF_N}^p\norm{f}_{\Hp}^p,  
\end{align*}
where the constant $C(\underline{\tilde{\phi}}^\text{exp})$ comes from the implicit constants in Lemma \ref{lem:hardycomparableseminorms}. Note that, once we fix the function $\Xi$ allowing to expand a periodic test function $\tilde{\phi}$ according to rule \eqref{e:rule_TtoR}, we have $C(\underline{\tilde{\phi}}^\text{exp})\norm{\underline{\tilde{\phi}}^\text{exp}}_{\fancyF_N}^p\lesssim  \norm{{\tilde{\phi}}}_{\fancyF_N}^p$, and therefore 
\begin{equation*}
|\int f(x)\tilde{\phi}(x)dx|^p\lesssim \norm{{\tilde{\phi}}}_{\fancyF_N}^p\norm{f}_{\Hp}^p.    
\end{equation*}
This shows that the inclusion $\Hp\subset \fancyD'(\T^n)$ is continuous.
\epf

\begin{prop}
For $0<p\leq 1$, $\Hp$ is a complete metric space with the metric given by $d(f,g)=\norm{f-g}_{\Hp}^p$.
\end{prop}
\bpf
Let $a_i$ be a Cauchy sequence in $\Hp$. Since $\Hp\subset \fancyD'$ continuously, $a_i$ converges to $a$ in $\fancyD'$. In view of uniqueness of limits in $\fancyD'$, it suffices now to show that $a \in\Hp$. Choose a subsequence 
$\tilde{a}_n$ of $a_i$ so that
\[
d(\tilde{a}_n,\tilde{a}_{n+1})<2^{-n}
\]
and define
$$f_n=\tilde{a}_{n+1}-\tilde{a}_n,$$
observe that $\sum \norm{f_n}_{\Hp}^p<\infty$. Since the partial sums of this series are Cauchy in $\Hp\subset \fancyD'$, the series converges in $\fancyD'$ to a distribution $f=a-\tilde a_0$.
We have
$$\left(m_\Psi f(x)\right)^p = \left(m_\Psi \sum f_n(x)\right)^p \leq (\sum m_\Psi f_n(x))^p\leq \sum (m_\Psi f_n(x))^p$$
because $p \le 1$,
and integrating it we obtain 
$$\norm{f}_{\Hp}^p\leq \sum \norm{f_n}_{\Hp}^p<\infty,$$
so $f\in \Hp$, and $a = f+\tilde a_0\in \Hp$.
\epf

\subsection{Atomic decomposition}\label{subsec:atomicdecomposition}
A different way to characterize Hardy spaces $\Hp$ is through atomic decomposition, which allows any function in $\Hp$ to be expressed as a linear combination of 'atoms'. i.e.\ functions with well-controlled properties in terms of size, support, and cancellation. This characterization provides an effective tool for analyzing functions in $\Hp$, without relying on maximal function definitions.

In the atomic decomposition, the support of an atom plays an important role. 
Since we are working simultaneously with periodic and non-periodic functions, we have naturally identified (a function on) $\T^n$ with its periodic extension, i.e.\ with (function on) $\R^n$ with periodicity. 
However, in what follows, instead of working with infinite copies (of the same function), it will be sometimes convenient to work with a single representative. This motivates the following notation, which formalises the choosing of the copy closest to the origin, while keeping its support connected.

\begin{defi}\label{def:Rnball}
Let $f$ be a periodic function supported on a ball $B_f\subsetneq \T^n$. We denote by $\overline{B}_f$ the ball in $\R^n$ such that $(\cup_{i\in\Z^n} \overline{B}_f)/\Z^n\equiv B_f\subset \T^n$, and such that its origin is closest to $0$. If this choice is not yet unique, we choose the ball whose origin has more coordinates nonnegative. If $\supp f= \T^n$, we set $\bar{B}_f=B(0,\sqrt{n})\subset \R^n$.
We identify $f$ with its representative supported on $\overline{B}_f$. We say that the ball $\overline{B}_f$ is associated to $f$.
\end{defi}

Let us remark that, according to the above definition, we understand that any periodic function $f$ is supported on a periodic ball, and we can always choose a representative ball in $\R^n$ that contains all the information of $f$ (see Figure \ref{figureatoms}).

\begin{defi}\label{def:smallball}
 We say that $B$ is a small ball if $|{B}|<C_{at}$, for a fixed constant $C_{at}$ small enough. In the following, we will take $C_{at}=\frac{1}{10}$.
\end{defi}

As in the case of $\Hp_{loc}(\R^n)$, cf \cite{goldberg1979}, we require that $\Hp$-atoms supported in small balls integrate to $0$, when tested with  polynomials (up to certain degree). Observe that for periodic functions supported on small balls, testing them with any polynomial $p$ is well-defined.  Furthermore, in such case, if for $f$ supported on ${\bar B_f}$ we have $\int_{\bar B_f} p(x) f =0$, then $\int_{\bar B_f+x_0} p(x) f =0$, so periodicity is not relevant.

Keeping in mind the above identification, we introduce 
\begin{defi}[$\Hp$-atoms]\label{def:atoms}
Let $a\in \fancyD'(\T^n)$ be a measurable function, and $\bar{B}_a$ its representative ball according to Definition \ref{def:Rnball}. We say that $a$ is an $\Hp$-atom associated to $\bar{B}_a$ if:
\begin{enumerate}
    \item $\norm{{a}}_{L^\infty}\leq |\bar{B}_{a}|^{-\frac{1}{p}}$.
    \item If $\bar{B}_{a}$ is a small ball, $\int_{\overline{B}_{a}} x^\beta {a}=0$ for all multiindices $ \beta$ with $|\beta|\leq n(p^{-1}-1)$. \label{itematomscancellations}
\end{enumerate} 
\end{defi}

Observe that one of the advantages of choosing the representative as in Definition \ref{def:Rnball} is that we can write the point 2.\ of Definition \ref{def:atoms}, while writing $\int_{\T^n}x^\beta a$ is problematic, because $x$ is not periodic.

\begin{prop}\label{prop:atomicbound}
 There is a uniform constant $C$, such that for all $\Hp$-atoms $a$ it holds $\norm{a}_{\Hp}\leq C$ (where the constant may depend on $d,n, \Psi$).
\end{prop}
\bpf
To simplify the exposition, we assume without loss of generality that $\Psi$ is supported on the unit ball.

Take an arbitrary $\Hp$-atom $a$, and let $\bar{B}_a$ be the associated ball according to Definition \ref{def:Rnball}.
Using the property 1.\ of Definition \ref{def:atoms}, and the fact that maximal function is pointwisely bounded by the $L^\infty$ norm, we have
\begin{equation}\label{eq:atomHpestimategeneral}
m_\Psi a(x)\leq |\bar{B}_a|^{-\frac{1}{p}}.  
\end{equation}
(i) Assume first that the ball $\bar{B}_a$ is not a small ball. Using the bound \eqref{eq:atomHpestimategeneral} and the fact that $|\bar{B}_a|\geq C_{at}$, we have
\begin{equation}\label{eq:Hpboundbigatoms}
\norm{a}_{\Hp}^p=\int_{\T^n}(m_\Psi a(x))^p \leq \int_{\T^n} |\bar{B}_a|^{-1}\leq C_{at}^{-{1}}.
\end{equation}
(ii) We assume now that $\bar{B}_a$ is a small ball, and we denote its center by $\bar{x}$. We define the hypercube $D=[-1/2,1/2]^n+\bar{x}\subset \R^n$ (see Figure \ref{figureatoms}).

\begin{figure}[ht]
    \centering
  \includegraphics[width=0.5\textwidth,trim=2cm 2cm 0cm 2cm,clip]{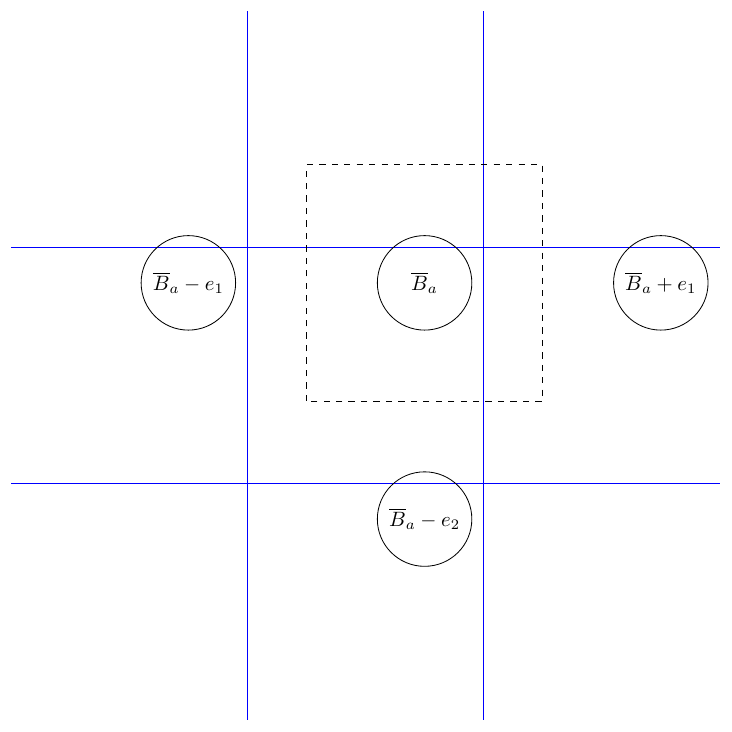}
    \caption{Picture of our domain. The periodic function $a$ is supported on the periodic ball denoted by circles. We select as representative the ball $\bar{B}_a\subset \R^n$, and we work on the domain $D$, whose borders are marked with a dashed line. The vertical lines correspond to $x_1=\pm \frac{1}{2}$, and the horizontal lines correspond to $x_2=\pm \frac{1}{2}$.}
    \label{figureatoms}
\end{figure}

We will use the bound \eqref{eq:atomHpestimategeneral}
for $x\in \bar{B}_a^*$, where $\bar{B}_a^*$ has the same center as $\bar{B}_a$ and twice its radius. To deal with $x\in D\setminus B_a^*$, we consider two cases.

(ii.1) Firstly, we treat the case $\zeta\leq \frac{1}{4}$. Since $\zeta$ and $|\bar{B}_a|$ are small enough, the support of $\Psi_\zeta(x-y)$ intersects only $\bar{B}_a$ in the set of balls $\{\bar{B}_a+k, k\in \Z^n\}$. Thus
$$(a*\Psi_\zeta)(x)=\int_{\bar{B}_a} a(y)\Psi_\zeta(x-y)dy.$$
We use the cancellation of moments (Property 3. of Definition \ref{def:atoms}) to write
$$\int_{\overline{B}_a} a(y)\Psi_\zeta(x-y)dy=\int_{\overline{B}_a}a(y)[\Psi_\zeta(x-y)-P_{x,\Psi_\zeta,d}(y)]dy, $$
where $P_{x,\Psi_\zeta,d}$ is the degree $d=\lfloor n(\frac{1}{p}-1)\rfloor$ Taylor polynomial of the function $y\mapsto \Psi_\zeta(x-y)$ expanded over $\bar{x}$, the center of $\overline{B}_a$. By properties of the Taylor expansion, we have
$$|\Psi_\zeta(x-y)-P_{x,\Psi_\zeta,d}(y)|\lesssim \frac{|y-\bar{x}|^{d+1}}{\zeta^{d+1+n}}.$$
Let $r$ be the radius of $\overline{B}_a$.
Recall $y\in \overline{B}_a, x\not\in \overline{B}_a^*$. 
In case $|x-y|\le \zeta$, which means in particular that $r\le \zeta$, we have 
$|x-\bar{x}| \le |x-y| + |\bar{x}-y| \le \zeta +r \le 2 \zeta$. In case $|x-y|> \zeta$ we have $\Psi_\zeta(x-y)=0$. 
Therefore we have for any $x\in D\setminus \overline{B}_a^*$ and any $\zeta<\frac{1}{4}$
\[
|(a*\Psi_\zeta)(x)| = \left| \int_{\overline{B}_a}a(y)[\Psi_\zeta(x-y)-P_{x,\Psi_\zeta,d}(y)]dy \right| \lesssim \norm{{a}}_{L^\infty} \left| \int_{\overline{B}_a}  \frac{r^{d+1}}{|x-\bar{x}|^{d+1+n}} dy\right|, 
\]
which gives via Property 2. of Definition \ref{def:atoms}

$$\sup_{0<\zeta<\frac{1}{4}}|a*\Psi_\zeta(x)| \lesssim r^{-\frac{n}{p}}\frac{r^{d+1+n}}{|x-\bar{x}|^{d+1+n}}.$$

Altogether, we have, using that $d=\lfloor n(\frac{1}{p}-1)\rfloor >n(\frac{1}{p}-1) -1$

\begin{equation}\label{eq:Hpboundsmallatomlocalmax}
\begin{aligned}
\int_{\T^2} \sup_{0<\zeta<\frac{1}{4}}|a*\Psi_\zeta(x)|^pdx=&\int_{\bar{B}_a^*} (m_\Psi a(x))^pdx+\int_{D\setminus \bar{B}_a^*} \sup_{0<\zeta<\frac{1}{4}}|a*\Psi_\zeta(x)|^pdx\\
\lesssim& \int_{\bar{B}_a^*} |\bar{B}_a|^{-1}+\int_{2r}^{\sqrt{n}}r^{-n} \frac{r^{p(d+1+n)}}{z^{p(d+1+n)-(n-1)}}dz\lesssim 1.
\end{aligned}
\end{equation}

(ii.2) We study now the case $x\in D\setminus \bar{B}_a^*$, $\zeta\geq \frac{1}{4}$. We note that, for any fixed $x$, the support of $\Psi_\zeta(x-y)$ may intersect more than one ball in the set of balls $\{\bar{B}_a+k, k\in \Z^n\}$. We denote $K_x$ as the set of $k\in \Z^n$ such that $(\supp \Psi_\zeta(x-y))\cap (\bar{B}_a+k)\neq \emptyset$.
Then $|K_x|\lesssim \zeta^n$. We write

$$(a*\Psi_\zeta)(x)=\sum_{k\in K_x}\int_{\bar{B}_a+k} a(y)\Psi_\zeta(x-y)dy=\sum_{k\in K_x}\int_{\bar{B}_a+k}a(y)[\Psi_\zeta(x-y)-P^k_{x,\Psi_\zeta,d}(y)]dy, $$
where $P^k_{x,\Psi_\zeta,d}$ is the degree $d=\lfloor n(\frac{1}{p}-1)\rfloor$ Taylor polynomial of the function $y\mapsto \Psi_\zeta(x-y)$ expanded over $\bar{x}+k$, the center of $\bar{B}_a+k$. By properties of the Taylor expansion, we have
$$|\Psi_\zeta(x-y)-P^k_{x,\Psi_\zeta,d}(y)|\lesssim \frac{|y-\bar{x}|^{d+1}}{\zeta^{d+1+n}}.$$
We bound
$$|(a*\Psi_\zeta)(x)|\lesssim \zeta^n r^{-\frac{n}{p}}r^{d+1+n}\zeta^{-(d+1+n)}\lesssim 1,$$
since $r<1,-\frac{n}{p}+d+1+n\geq 0$ and $\zeta\geq \frac{1}{4}, n-d-1-n\leq 0$. This implies 

\begin{equation}\label{eq:Hpboundsmallatomglobalmax}
\begin{aligned}
\int_{\T^2} \sup_{\frac{1}{4}<\zeta<\infty}|a*\Psi_\zeta(x)|^pdx=&\int_{\bar{B}_a^*} (m_\Psi a(x))^pdx+\int_{D\setminus \bar{B}_a^*} \sup_{\frac{1}{4}<\zeta<\infty}|a*\Psi_\zeta(x)|^pdx\\
\lesssim& \int_{\bar{B}_a^*} |\bar{B}_a|^{-1}+\int_{D\setminus \bar{B}_a^*} 1 dx\lesssim 1.
\end{aligned}
\end{equation}
From $\eqref{eq:Hpboundbigatoms},\eqref{eq:Hpboundsmallatomlocalmax}$ and $\eqref{eq:Hpboundsmallatomglobalmax}$, we conclude

\begin{align*}
\norm{a}_{\Hp}^p\lesssim 1.
\end{align*}
\epf
\begin{prop}\label{prop:sumatomsinHp}
If $\{a_k\}$ is a collection of $\Hp$ atoms and $\{\lambda_k\}$ is a sequence of numbers with $\sum|\lambda_k|^p<\infty,$ then the series 
$$f=\sum a_k\lambda_k$$
converges in $\fancyD'(\T^n)$, and its sum belongs to $\Hp$, with $\norm{f}_{\Hp}\leq c(\sum\lambda_k^p)^{\frac{1}{p}}$.
\end{prop}
\bpf
If the sum is finite, we have 
$$m_\Psi f\leq \sum |\lambda_k|m_\psi a_k,$$
and 
$$\Big(\sum |\lambda_k|m_\psi a_k\Big)^p\leq \sum |\lambda_k|^p(m_\Psi a_k)^p,$$
since $0<p\leq 1$. Integrating these inequalities and using Proposition \ref{prop:atomicbound}, we obtain 
$$\norm{f}_{\Hp}\leq c(\sum\lambda_k^p)^{\frac{1}{p}}.$$
This shows the unconditional convergence of the series in the $\Hp$ metric, thus also in $\fancyD'$.
\epf

\begin{cor}\label{cor:LinftyinHp}
A function $f\in L^\infty(\T^n)$ satisfying condition \ref{itematomscancellations} in Definition \ref{def:atoms}, satisfies the bound 
$$\norm{f}_{\Hp}\leq C|\bar{B}_f|^{\frac{1}{p}}\norm{f}_{L^\infty}.$$
\end{cor}
\bpf
Taking $\lambda=\norm{f}_{L^\infty}|\bar{B}_f|^\frac{1}{p},$ it is clear that $\lambda^{-1} f$ is an atom. Therefore, we can apply Proposition \ref{prop:sumatomsinHp} to $f=(\lambda^{-1}f)\lambda$, giving the bound 
$$\norm{f}_{\Hp}\lesssim \lambda = \norm{f}_{L^\infty}|\bar{B}_f|^\frac{1}{p}.$$ 
\epf
We are interested now in bounding the $\Hp$ norm of a concentrated function not satisfying \ref{itematomscancellations} in Definition \ref{def:atoms}. One can do it straightforwardly by using the embedding $\Hp\subset L^q$ for any $0<p\leq 1<q$, but the resulting bound turns out to be insufficient for our purposes.

Instead, we will correct locally the moments of the function,  without increasing its support, and consequently obtain better bounds.

For this, we want to be able to reproduce any value of a moment with a function.
Let $\alpha$ be a multiindex, $|\alpha|\le N$. For a compactly supported function $f$, we denote its $\alpha$-moment by $m_\alpha$, i.e.\
$m_\alpha:=\int_{\R^n} x^\alpha f(x)dx$. Denote $|m|:=\max_{|\alpha|\leq N}\left|\int_{\R^n} x^\alpha f(x) dx \right|.$ Observe that the set $\{m_\alpha | \; |\alpha|\le N\} =\R^{{\left(\begin{smallmatrix} N+n \\ n \end{smallmatrix}\right)}}$.
 Let us  recall \cite[Lemma 3.11]{buckmodena2024compactly}

\begin{lemma}\label{lem:BMpolynomia}
For every $N\in \N$, there is a degree $N$ polynomial $p_N$ on $\R$ such that the following holds. For all $\eps\in(0,1]$ and $\bar{x}\in\R^n$, there is a linear operator $m \to L_{\eps,\bar{x},m}$ that maps each choice of order $\leq N$ moments $m\in \R^{{\left(\begin{smallmatrix} N+n \\ n \end{smallmatrix}\right)}}$ to a function $L_{\eps,\bar{x},m}\in C^\infty_c(B(\bar{x},\eps))$,
so that 
$$\int_{\R^n} x^\alpha L_{\eps,\bar{x},m}(x)dx=m_{\alpha} \text{ for any multiindex } |\alpha|\leq N,$$
and
$$\norm{L_{\eps,\bar{x},m}}_{L^\infty}\leq \frac{p_N(|\bar{x}|)}{\eps^{n+N}}|m|.$$
\end{lemma}

Now we can control the $\Hp$-norm of concentrated functions not satisfying \ref{itematomscancellations} in Definition \ref{def:atoms} by using the following proposition. It is a 'periodic version' of \cite[Proposition 3.9]{buckmodena2024compactly}. 

\begin{prop}\label{prop:Hpboundnonvanishingatoms}
Let $0<p\leq 1$. Let $f\in L^\infty(\T^n)$ be a function, and $\bar{B}_f$ its representative ball according to Definition \ref{def:Rnball}. Assume that $\bar{B}_f$ has center $\bar{x}$ and radius $\eps$, with $0<\eps<\frac{1}{2}$. Let $N= \lfloor n(\frac{1}{p}-1)\rfloor$. Then 
\begin{align*}
\norm{f}_{\Hp}^p\leq &C(p,n)\big(\eps^n\norm{f}_{L^\infty}^p
+|\log\eps|\max_{|\alpha|\leq N}\left|\int_{\bar{B}_f} x^\alpha f^\text{ext}(x) dx \right|^p\big).
\end{align*}
\end{prop}
\bpf
The idea is to add a local corrector $L$ to $f$ so that it satisfies the condition \ref{itematomscancellations} in Definition \ref{def:atoms}. We will do it in one period centered at $\bar{x}$, and extend the correction periodically.
We choose a family of shrinking radii 
$$\eps_i=2^{-i}\text{ for } 1\leq i<K,$$
where $K=-\lfloor \log_2(\eps)\rfloor\lesssim |\log\eps|$, and we set $\eps_K=\eps$, so that
$$\frac{1}{2}=\eps_1>\eps_2>\dots>\eps_K=\eps.$$
Using the functions from Lemma \ref{lem:BMpolynomia}, we decompose
\begin{align*}
f=f-L_{\eps,\bar{x},m}+\sum_{i=1}^{K-1}L_{\eps_{i+1},\bar{x},m}-L_{\eps_i,\bar{x},m}+L_{\eps_1,\bar{x},m},
\end{align*}
and we extend this decomposition periodically. We use Corollary \ref{cor:LinftyinHp} and Lemma \ref{lem:BMpolynomia} to estimate
$$\norm{f-L_{\eps,\bar{x},m}}_{\Hp}^p\lesssim \eps^n \left(\norm{f}_{L^\infty}+\frac{1}{\eps^{n+N}}|m|\right)^p
\lesssim \eps^n \norm{f}_{L^\infty}^p+|m|^p,$$
since $\eps<1$ and $N\leq n(p^{-1}-1)\Rightarrow p(n+N)\leq n$.
We use again Corollary \ref{cor:LinftyinHp} and Lemma \ref{lem:BMpolynomia} to estimate
$$\sum_{i=1}^{K-1}\norm{L_{\eps_{i+1},\bar{x},m}-L_{\eps_i,\bar{x},m}}_{\Hp}^p\lesssim \sum_{i=1}^{K-1}\frac{2^{-in}}{2^{-(i+1)p(n+N)}}|m|^p\lesssim \sum_{i=1}^{K-1}2^{-i(n-p(n+N))}|m|^p.$$
We note that the power of $2$ in the previous series is always negative. Therefore, 
$$\sum_{i=1}^{K-1}\norm{L_{\eps_{i+1},\bar{x},m}-L_{\eps_i,\bar{x},m}}_{\Hp}^p\lesssim K|m|^p\lesssim |\log\eps||m|^p.$$
Finally,
$$\norm{L_{\eps_{1},\bar{x},m}}_{\Hp}^p\lesssim \norm{L_{\eps_{1},\bar{x},m}}_{L^\infty}^p\lesssim |m|^p,$$
concluding the proof.
\epf

\begin{prop}\label{prop:atomicdecomposition}
Every $f\in \Hp$ can be written as $f=\sum_k \lambda_k a_k$, where $a_k$ are $\Hp$ atoms, and $\lambda_k$ are real numbers. The sum converges to $f$ strongly in $\Hp$, and 
\begin{equation}\label{eq:atomicdecomposition}
\sum_k |\lambda_k|^p\lesssim \norm{f}_{\Hp}^p.
\end{equation}
\end{prop}
\bpf
The idea of the proof for the analogous result in $\Hp(\R^n)$ \cite{stein1993harmonic} is to use the Whitney covering lemma for each set $O^j=\{x:\mathcal{M}f>2^j\}$, and decompose $f=g_j+b_j$, with $\norm{g_j}_{L^\infty}\leq c 2^j$, and $b_j$ supported on $O^j$ and having the property \ref{itematomscancellations} in Definition \ref{def:atoms}. Then, it is observed that $\lim_{j\to\infty} g_j=f$ in $\Hp(\R^n)$, and an atomic decomposition of $\sum_j g_{j+1}-g_j$ is constructed. Lastly, equation \eqref{eq:atomicdecomposition} is checked.
 A complete proof in $\R^n$ can be found in \cite{stein1993harmonic}. The proof for $\Hp\cap L^1_{loc}(\R^n)$ functions can be found in pages $101-105$ together with $107-109$. The proof for a general $\Hp(\R^n)$ element is completed in pages $110-111$. 
 
 To adapt the proof from $\Hp(\R^n)$ to $\Hp$, we observe that we can decompose $f=g_j+b_j$ with the same procedure if $|O^j|<C$, for $C>0$ small enough. Since $|O^{j}|$ is decreasing and converging to $0$, it is enough to fix $j_0$ such that $|O^{j_0}|<C$. We choose the minimum possible value of $j_0$. Then, we can write 
 $$f=g_{j_0}+\sum_{j=j_0}^\infty g_{j+1}-g_j.$$

To obtain an atomic decomposition of $g_{j_0}$, we note that $c^{-1}2^{-j_0}g_{j_0}$ is an atom. To keep the bound \eqref{eq:atomicdecomposition}, we need to show that $2^{j_0}p\lesssim \norm{f}_{\Hp}^p$. Using that $j_0$ is minimal, we observe

$$2^{j_0}p\leq |O^{j_0-1}|C^{-1}2^p 2^{(j_0-1)p}\leq C^{-1}2^p \norm{f}_{\Hp}^p.$$
We can obain an atomic decomposition of $\sum_{j=j_0}^\infty g_{j+1}-g_j$ satisfying \eqref{eq:atomicdecomposition} proceeding as in $\R^n$. Combined with the decomposition of $g_{j_0}$, we have an atomic decomposition of $f$.

\epf

\begin{rem}\label{rem:highercancellationatoms}
It is possible to impose the property 2. in Definition \ref{def:atoms} for arbitrarily higher multiindices in the atomic decomposition given by Proposition \ref{prop:atomicdecomposition}. This does not add any difficulty to the proof.
\end{rem}

\subsection{Proof of Proposition \ref{prop:H1inL2}}
Fix any function $\varphi\in L^2$ with $\norm{\varphi}_{L^2}\leq 1$. It can be written as $\varphi=\curl g_1+\na g_2$ for some $g_1,g_2\in H^1$, with $\norm{g_1}_{H^1},\norm{g_2}_{H^1}\lesssim 1$. 

In the sense of distributions,
$$|\langle f, \varphi\rangle|\leq |\langle f, \curl g_1\rangle|+|\langle f, \na g_2\rangle|=|\langle \curl f, g_1\rangle|.$$
In two dimensions $H^1\hookrightarrow (\mathcal{H}^1)^*$, so 
$$|\langle f, \varphi\rangle|\lesssim \norm{\curl f}_{\mathcal{H}^1}.$$
Since $\varphi$ was arbitrary, this shows that 
$$\norm{f}_{L^2}\lesssim \norm{\curl f}_{\mathcal{H}^1}.$$

\subsection{Singular Integral Operators}
In this subsection, we show that singular integral operators are bounded on $\Hp$, as it occurs in $\Hp(\R^n)$. Since we are using  simultaneously periodic and non-periodic functions, we begin by clarifying some aspects regarding the Fourier transform. 
Given $f\in \fancyS(\R^n)$, and its periodic version ${f^\text{per}}(\cdot)=\sum_{k\in\Z^n} f(\cdot+k)\in \fancyD(\T^n)$, we have
\begin{align*}
\hat{f}(\xi)&=\int_{\R^n} e^{-i2\pi\xi \cdot x}f(x)dx,\quad &&\xi\in\R^n, \\
\hat{{f}}^\text{per}(\xi)&=\int_{\T^n} e^{-i2\pi\xi \cdot x}{f^\text{per}}(x)dx,\quad &&\xi\in\Z^n.
\end{align*}
Since $e^{-i2\pi\xi \cdot x}$ is also periodic, it is direct that $\hat{f}(\xi)=\hat{{f}}^\text{per}(\xi)$ $\forall \xi\in\Z^n.$ 

We consider now a periodic function $g$,
whose periodic extension to $\R^n$ is denoted by $g^\text{ext}$. The $\R^n$ convolution $f*g^\text{ext}$ is equal to the $\T^n$ convolution ${f}^\text{per}*_{\T^n}g$, so we have the identity
\begin{equation}\label{eq:convolutionidentity}
\begin{aligned}
f*g^\text{ext}(x)={f}^\text{per}*_{\T^n}g(x)=\sum_{\xi\in\Z}e^{i2\pi\xi\cdot x}\hat{f}^\text{per}(\xi)\hat{{g}}(\xi)=\sum_{\xi\in\Z}e^{i2\pi\xi\cdot x}\hat{f}(\xi)\hat{{g}}(\xi).    
\end{aligned}
\end{equation}

\begin{prop}\label{prop:PboundedHp}
Let $\ProL$ be the Leray projector, $0<p<\infty$, and $f\in \Hp$. Then, $$\norm{\ProL f}_{\Hp}\lesssim  \norm{ f}_{\Hp}.$$
\end{prop}
\bpf
The case $1<p<\infty$ is a classic result. We prove the case $0<p\leq 1.$ Since $\ProL$ is a linear operator, it is enough to prove that the Leray projector of an atom is bounded in $\Hp$.

We first prove the case where $a$ is an atom supported on a big ball. Then, there exists $C$ independent of $a$ such that $a\leq C$, which implies

\begin{equation*}
\begin{aligned}
\norm{\ProL a}_{\Hp}=&\left(\int_{\T^n} (m_\psi \ProL a)^p\right)^\frac{1}{p}
\leq \left(\int_{\T^n} (m_\psi \ProL a)^2\right)^{\frac{1}{2}}
\lesssim \left(\int_{\T^n} (\ProL a)^2\right)^{\frac{1}{2}}
\lesssim \left(\int_{\T^n} a^2\right)^{\frac{1}{2}} \\
&\leq  \left(\int_{\T^n} a^p C^{2-p}\right)^{\frac{1}{2}} \leq C^{\frac{2-p}{2}}\left(\int_{\T^n} (m_\psi a)^p \right)^{\frac{1}{2}} \lesssim \norm{a}_{\Hp}^{\frac{p}{2}}\lesssim 1.
\end{aligned}
\end{equation*}

We deal now with the case where $a$ is an atom supported on a small ball. We adapt the proof of Theorem 5, case II of \cite{Fan2001}. Using Remark \ref{rem:highercancellationatoms}, we can assume that $a$ satisfies property 2. in Definition \ref{def:atoms} for any multiindex $|\beta|\leq n(p^{-1}+1)$. Besides, we can assume without loss of generality, that $a\in C^\infty(\T^2)$, and that its representative ball  $\bar{B}_a$  according to Definition \ref{def:Rnball} is centered at the origin. We denote the Fourier expansion of $a$ by
$$a=-\sum_{k\in \Z^n}\alpha_k e^{2\pi i k\cdot x}. $$
Thus, we have 
$$\ProL a= -\sum_{k\in \Z^n\setminus 0}\left(I-\frac{k\otimes k}{|k|^2}\right)\alpha_k e^{2\pi i k\cdot x}, $$
where we have used that $\int a=0$.

We let
$$\Xi(x)\coloneqq \prod_{j=1}^n(1-4x_j^2)_+,$$
where $2x_+=x+|x|$.
We denote $\Xi^\frac{1}{N}(x)=\Xi(x/N).$ We denote by $\ProL_{\R^n}$ the Leray projector on $\R^n$. We let

$$\mathcal{E}_N(x)= \Xi^\frac{1}{N}(x)(\Psi_\zeta*\ProL a)^\text{ext}(x)-\Psi_\zeta*\ProL_{\R^n}(\Xi^{\frac{1}{N}}a^\text{ext})(x).$$
By checking the Fourier transform, we see that
\begin{align*}
\mathcal{E}^t_N(x)=\sum_{k\in \Z^n\setminus 0}\alpha^t_ke^{2\pi i k\cdot x}\int_{\R^n} \hat{\Xi}(\xi)&\left[\hat{\Psi}(\zeta (k+\frac{\xi}{N}))\left( I-\frac{(k+\frac{\xi}{N})\otimes (k+\frac{\xi}{N})}{|k+\frac{\xi}{N}|^2}\right)\right. \\
& - \left.
\hat{\Psi}(\zeta k)\left(I-\frac{k\otimes k}{|k|^2}\right)
\right]^te^{\frac{2\pi i \xi\cdot x}{N} }d\xi,    
\end{align*}
where $v^t$ is the transpose of the vector or matrix $v$.
We fix $R>0$. Since $\alpha_k$ decays rapidly, and the terms in $[\cdot ]$ are bounded and converging to $0$ if $N\to \infty$, 
\begin{equation}\label{eq:EpsNto0}
\mathcal{E}_N(x)\to 0 \quad \text{uniformly for $x\in \R^n, \zeta\in (0,R]$ as $N\to \infty.$}
\end{equation}

For large $N$, we have 
\begin{align*}
\|\sup_{0<\zeta\leq R} |\Psi_\zeta*\ProL a|\|^p_{L^p(\T^n)}\lesssim  &  N^{-n}\int_{[-N/2,N/2]^n}\sup_{0<\zeta\leq R} |  \Xi^\frac{1}{N} (\Psi_\zeta*\ProL a(x))^\text{ext}|^pdx \\
\lesssim &   N^{-n}\int_{\R^n}\sup_{0<\zeta\leq R} | \Psi_\zeta*\ProL_{\R^n}  (\Xi^\frac{1}{N}  a^\text{ext})(x)|^pdx + o(1),
\end{align*}
where we have used the property \eqref{eq:EpsNto0}. 
Using that $\norm{\ProL_{\R^n} f}_{\Hp(\R^n)}\lesssim \norm{f}_{\Hp(\R^n)}$, we obtain

\begin{align}\label{eq:LerProbound0}
\|\sup_{0<\zeta\leq R} |\Psi_\zeta*\ProL a|\|^p_{L^p(\T^n)}\lesssim  N^{-n}\|\Xi^\frac{1}{N}  a^\text{ext}\|_{\Hp(\R^n)}^p + o(1).
\end{align}

By definition, we have 
\begin{equation}\label{eq:LerProbound1}
\begin{aligned}
\|\Xi^\frac{1}{N}  a^\text{ext}\|_{\Hp(\R^n)}^p = & \int_{\R^n} \sup_{0<\zeta<\infty}\left|\int_{\R^n}\prod_{j=1}^n(1-4x_j^2/N^2)_{+}a^\text{ext}(x)\Psi_\zeta(y-x)dx \right|^pdy \\
= & \int_{\R^n} \sup_{0<\zeta<\infty}\left|\int_{|x_j|<N/2}\prod_{j=1}^n(1-4x_j^2/N^2)_{+}a^\text{ext}(x)\Psi_\zeta(y-x)dx \right|^pdy.    
\end{aligned}
\end{equation}

We choose $N=2m+1$, $m\in \N$. Then, up to a set of measure $0$, the set $\{|x_j|<N/2\}$ is the union of the disjoint sets $Q_k=\{(-1/2,1/2)^n+k\},$ where $k\in \Z^n, |k_j|\leq m$. We can bound \eqref{eq:LerProbound1} by
\begin{equation}\label{eq:LerProbound2}
 \|\Xi^\frac{1}{N}  a^\text{ext}\|_{\Hp(\R^n)}^p \lesssim \sum_{|k_j|\leq m} \int_{\R^n} \sup_{0<\zeta<\infty}\left|\int_{Q_k}\prod_{j=1}^n(1-4x_j^2/N^2)_{+}a^\text{ext}(x)\Psi_\zeta(y-x)dx \right|^pdy.   
\end{equation}

On $Q_k$, $\prod_{j=1}^n(1-4x_j^2/N^2)_{+}$ is a polynomial of degree $2n$ which is bounded by $1$. Therefore, $\prod_{j=1}^n(1-4x_j^2/N^2)_{+}a^\text{ext}(x)\mathbbm{1}_{Q_k}(x)$ is an $\Hp(\R^n)$-atom, so 
$$\int_{\R^n} \sup_{0<\zeta<\infty}\left|\int_{Q_k}\prod_{j=1}^n(1-4x_j^2/N^2)_{+}a^\text{ext}(x)\Psi_\zeta(y-x)dx \right|^pdy\lesssim 1$$
for any $k\in \Z^n, |k_j|\leq m$. We use \eqref{eq:LerProbound2} to bound
$$\|\Xi^\frac{1}{N}  a^\text{ext}\|_{\Hp(\R^n)}^p\lesssim N^n,$$
showing that 
$\liminf_{N\to \infty} N^{-n}\|\Xi^\frac{1}{N}  a^\text{ext}\|_{\Hp(\R^n)}^p\lesssim 1$. We use \eqref{eq:LerProbound0} to bound
$$\|\sup_{0<\zeta\leq R} |\Psi_\zeta*\ProL a|\|^p_{L^p(\T^n)}\lesssim  1,$$
uniformly on $R$. We conclude 
$$\norm{\ProL a}_{\Hp}^p=\|\sup_{0<\zeta<\infty} |\Psi_\zeta*\ProL a|\|^p_{L^p(\T^n)}\lesssim  1.$$
\epf

\begin{rem}
Proposition \ref{prop:PboundedHp} is the only result from this subsection that we need for the main part. From its proof, one can also deduce the boundedness in $\Hp$ of other singular operators that are bounded on $\Hp(\R^n)$, provided their Fourier symbols are well controlled.

We now prove a complementary result, which may be of independent interest: it shows that the boundedness of singular operators can also be established by controlling their singularity in the physical variables.
\end{rem}

\begin{lemma}\label{lem:kernelbounds}
Let $\Psi\in\fancyS(\R^n)$, $M\in\N$ and $K$ be a periodic function satisfying
\begin{equation}\label{eq:singularkernelassumptions}
\begin{aligned}
 |\hat{K}(\xi)|\lesssim& 1,\quad &&\forall \xi\in\Z^n,\\
|\na^\beta K(x)|\lesssim& |x|^{-n-|\beta|}, \quad &&\forall |\beta|\leq M, \quad x\in \left[-\frac{1}{2},\frac{1}{2}\right]^n.   
\end{aligned}
\end{equation}

Then, the kernel
$$K_\zeta=\Psi_\zeta*K$$
satisfies estimates \eqref{eq:singularkernelassumptions}, with bounds that depend on $\Psi$ but can be chosen independent of $\zeta$.
\end{lemma}
\bpf
From the identity \eqref{eq:convolutionidentity} and the assumption $|\hat{K}(\xi)|\lesssim 1$, it is enough to bound the Fourier coefficients of $\Psi_\zeta$. By scaling of the Fourier transform, we have 
$$\widehat{\Psi_\zeta}(\xi)=\widehat{\Psi}(\zeta\xi),$$
so the property $|\widehat{K_\zeta}(\xi)|\lesssim 1$ follows from the boundedness of $\widehat{\Psi}$.

Without loss of generality, we assume that $\supp \Psi\subset B(0,1)$. We fix $x\in \left[-\frac{1}{2},\frac{1}{2}\right]^n$, which covers a whole period. Finally, we fix a multiindex $\beta$ with $|\beta|\leq M$. 

For $0<\zeta<\frac{|x|}{2}$, we bound
$$|\Psi_\zeta*\na^\beta K(x)|\leq \norm{\na^\beta K}_{L^\infty(B(x,\frac{|x|}{2}))}\lesssim |x|^{-n-|\beta|}.$$
For $|x|\lesssim \zeta$, we first observe 
$$\na^\beta\Psi_\zeta*K(x)=\zeta^{-|\beta|}(\na^\beta\Psi)_\zeta*K(x),$$
and $\zeta^{-\beta}\lesssim |x|^{-\beta}$ by assumption, so it is enough to prove the bound for $\beta=0$.
By \eqref{eq:convolutionidentity}, we have
$$\Psi_\zeta*K(x)=\sum_{\xi\in\Z^n}e^{i2\pi\xi x}\widehat{\Psi_\zeta}(\xi)\hat{K}({\xi}),$$
so 
\begin{equation}\label{eq:fouriertestbound1}
|\Psi_\zeta*K(x)|\lesssim \sum_{\xi\in\Z^n}|\widehat{\Psi_\zeta}(\xi)|=\sum_{\xi\in\Z^n}|\widehat{\Psi}(\zeta\xi)|.
\end{equation}

Since $\Psi\in\fancyS(\R^n)$, there is a function $g\in \fancyS(\R^n)$ with radial symmetry and strictly decreasing such that
$$|\widehat{\Psi}(\xi)|\leq g(\xi),$$
giving 

\begin{equation}\label{eq:fouriertestbound2}
\sum_{\xi\in\Z^n}|\widehat{\Psi}(\zeta\xi)|\leq \sum_{\xi\in\Z^n} g(\zeta\xi) \lesssim g(0)+\int_{\R^n}g(\zeta\xi)d\xi= g(0)+\zeta^{-n}\|g \|_{L^1}\lesssim |x|^{-n}.
\end{equation}
The use of the auxiliary function $g$ is not strictly necessary, but its properties significantly simplify the proof of inequality $\sum_{\xi\in\Z^n} g(\zeta\xi) \lesssim g(0)+\int_{\R^n}g(\zeta\xi)d\xi$.
\epf

\begin{prop}\label{prop:singularoperatorsHp}
Let $0<p<1$ and $K$ be a periodic function satisfying
\begin{align*}
|\hat{K}(\xi)|\lesssim& 1,\quad &&\forall \xi\in\Z^n,\\
|\na^\beta K(x)|\lesssim& |x|^{-n-|\beta|}, \quad &&\forall |\beta|\leq \lfloor n(\frac{1}{p}-1)\rfloor+1, \quad x\in \left[-\frac{1}{2},\frac{1}{2}\right]^n.
\end{align*}
Then, it holds
$$\norm{K*_{\T^n} f}_{\Hp}\lesssim \norm{f}_{\Hp}.$$
\end{prop}
\bpf
By Proposition \ref{prop:atomicdecomposition}, it is enough to show that $\norm{K*_{\T^n} a}_{\Hp}\lesssim \norm{a}_{\Hp}$ for an atom $a$, with a constant independent of the atom. 
We recall that both $m_\Psi$ and $K*_{\T^n}$ are bounded operators from $L^2$ into itself.

If $a$ is an atom supported on a big ball, we use Hölder inequality to bound
$$\int (m_\Psi(K*_{\T^n} a))^p\leq \Big(\int (m_\Psi (K*_{\T^n}a))^2\Big)^\frac{2}{p}\Big(\int 1\Big)^\frac{2-p}{2}\lesssim \Big(\int  a^2\Big)^\frac{2}{p}\Big(\int 1\Big)^\frac{2-p}{2}\lesssim 1,$$
since atoms supported on big balls are bounded from above, and the domain has finite size. 

We now assume that the atom $a$ has a small ball $\bar{B}_a\subset \R^n$ as its representative ball according to Definition \ref{def:Rnball}. We denote by $\bar{x}$ the center of $\bar{B}_a$. As we did in the proof of Proposition \ref{prop:atomicbound}, we will work in the domain $D=[-1/2,1/2]^n+\bar{x}\subset \R^n$ (see Figure \ref{figureatoms}). We also denote by $\bar{B}^*_a$ the ball with the same center as $\bar{B}_a$ and twice its radius. 

We split $D=\bar{B}^*_a\cup D\setminus \bar{B}^*_a$. In $\bar{B}^*_a$, we proceed as before to bound 

\begin{equation}\label{eq:SIOest1}
\begin{aligned}
\int_{\bar{B}^*_a} (m_\Psi(K*_{\T^n} a))^p&\leq \Big(\int_{\bar{B}^*_a} (m_\Psi (K*_{\T^n} a))^2\Big)^\frac{p}{2}\Big(\int_{\bar{B}^*_a} 1\Big)^\frac{2-p}{2}\\
&\lesssim \Big(\int_{\bar{B}^*_a}  a^2\Big)^\frac{p}{2}\Big(\int_{\bar{B}^*_a} 1\Big)^\frac{2-p}{2}\leq (|\bar{B}^*_a||\bar{B}_a|^{-\frac{2}{p}})^\frac{p}{2}|\bar{B}^*_a|^\frac{2-p}{2}\lesssim 1.
\end{aligned}
\end{equation}

For any $x\in D\setminus \bar{B}^*_a$, we have
\begin{align*}
K*_{\T^n} a*\Psi_\zeta(x)=&\int_{\R^n}\Psi_\zeta(x-y)\Big(\int_{\bar{B}_a}K(y-z)a(z)dz\Big)dy\\
=&\int_{\bar{B}_a}a(z)\Big(\int_{\R^n}\Psi_\zeta(x-y)K(y-z)dy\Big)dz
=\int_{\bar{B}_a}a(z)K_{\zeta}(x-z)dz,
\end{align*}
where $K_{\zeta}(x)=\int_{\R^n}\Psi_\zeta(x-y)K(y)dy$.

We denote $P_{x,\zeta}$ the degree $d=\lfloor n(\frac{1}{p}-1)\rfloor$ Taylor polynomial of the function $z\mapsto K_{\zeta}(x-z)$ expanded at $\bar{x}$. Because of property \ref{itematomscancellations}. in Definition \ref{def:atoms},we have
\begin{align*}
K*_{\T^n}  a*\Psi_\zeta(x)= &\int_{{\overline{B}_a}}a(z)(K_{\zeta}(x-z)-P_{x,\zeta}(z))dz.
\end{align*}
Since $z\in \bar{B}_a$ and $x\in D\setminus \bar{B}^*_a$, by properties of the Taylor expansion and Lemma \ref{lem:kernelbounds}, we have 
$$|K_{\zeta}(x-z)-P_{x,\zeta}(z)|\lesssim \frac{|z-\bar{x}|^{d+1}}{|x-\bar{x}|^{d+1+n}}.$$
Denoting as $r$ the radius of $\bar{B}_a$ and using $|a|\lesssim r^{-\frac{n}{p}}$, we have

\begin{equation*}
\begin{aligned}
|K*_{\T^n}  a*\Psi_\zeta(x)|\lesssim & \int_{{\overline{B}_a}}r^{-\frac{n}{p}} \frac{|z-\bar{x}|^{d+1}}{|x-\bar{x}|^{d+1+n}} dz 
\lesssim \int_{{\overline{B}_a}}r^{-\frac{n}{p}} \frac{r^{d+1}}{|x|^{d+1+n}} dz \lesssim  \frac{r^{-\frac{n}{p}+d+1+n}}{|x|^{d+1+n}}
\end{aligned}
\end{equation*}
uniformly on $\zeta$, and therefore

\begin{equation}\label{eq:SIOest2}
\begin{aligned}
\int_{D\setminus \bar{B}^*_a} (m_\Psi(K*_{\T^n}  a))^p\lesssim r^{p(-\frac{n}{p}+d+1+n)} \int_r^{\sqrt{n}} \frac{|x|^{n-1}}{|x|^{p(d+1+n)}} d|x|\lesssim r^{p(-\frac{n}{p}+d+1+n)} \left.\frac{r^{n}}{r^{p(d+1+n)}}\right.
\lesssim 1,
\end{aligned}
\end{equation}
where we used that $n-p(d+1+n)<0$.
Combining \eqref{eq:SIOest1} and \eqref{eq:SIOest2}, we conclude the case for atoms supported on small balls.
\epf

\bibliographystyle{abbrv}
\bibliography{NonuniquenessHardy}
\end{document}